\documentclass[11pt]{amsart}
\usepackage{fullpage}
\usepackage{adjustbox}
\usepackage{enumerate}
\usepackage{ae, amsfonts, amsmath,euscript}
\usepackage{amsthm}
\usepackage{comment}
\usepackage{amssymb,array}
\usepackage{enumitem}
\usepackage{booktabs}

\usepackage[usenames]{color}

\usepackage[colorlinks=true,
linkcolor=webgreen,
filecolor=webbrown,
citecolor=webgreen]{hyperref}
\definecolor{webgreen}{rgb}{0,.5,0}
\definecolor{webbrown}{rgb}{.6,0,0}

\usepackage{url}
\usepackage{float}
\usepackage{graphics}
\usepackage{latexsym}
\usepackage{epsf}

\newtheorem{theorem}{Theorem}[section]

\newtheorem{lemma}[theorem]{Lemma}
\newtheorem{corollary}[theorem]{Corollary}

\newtheorem{question}[theorem]{Question}

\theoremstyle{definition}
\newtheorem{example}[theorem]{Example}
\newtheorem{remark}[theorem]{Remark}

\newcommand{\gen}[1]{\langle #1 \rangle}
\newcommand{\unone}[1]{#1\cup\{1\}}

\newcommand{\Z}{\mathbb{Z}}

\newcommand{\C}{\mathbb{C}}
\newcommand{\Q}{\mathbb{Q}}

\newcommand{\inv}[1]{#1^{-1}}

\newcommand{\CT}{\mathrm{C}}
\newcommand{\invCT}{\mathrm{K}}

\newcommand{\drawCyclic}[2]
{
\begin{tikzpicture}
\def\n{#1};
\def\r{#2};
\def\minsz{27pt};
\def\sc{0.5};
\foreach \i in {1,2,...,\n} {
	\ifthenelse{\i = \n}
	{
		\node[circle,draw,fill=green,minimum size = \minsz,scale=\sc] (pos\i)  at ({\r*cos(\i*360/\n)}, {\r*sin(\i*360/\n)}) {$1$};
	}
	{
	\ifthenelse{\i = 1}
		{
		\node[circle,draw,fill=green,minimum size = \minsz,scale=\sc] (pos\i)  at ({\r*cos(\i*360/\n)}, {\r*sin(\i*360/\n)}) {$x$};
		}
		{
		\node[circle,draw,fill=green,minimum size = \minsz,scale=\sc] (pos\i)  at ({\r*cos(\i*360/\n)}, {\r*sin(\i*360/\n)}) {$x^{\i}$};
		}
	}
}
\foreach \i [evaluate=\i as \ione using {int(mod(\i,\n)+1)}]  in {1,2,...,\n} {
	\draw[->,blue] (pos\i) -- (pos\ione);
}
\end{tikzpicture}
}

\newcommand{\mcl}{\mathcal{L}}

\newcommand{\Mat}{\mathbb{M}}

\usepackage[table]{xcolor}
\usepackage{algorithm,algpseudocode}
\usepackage{todonotes}
\usepackage{float}
\usepackage{caption,subcaption}
\usepackage{indentfirst}
\usepackage{blkarray}
\usepackage{eqnarray}
\usepackage{amsmath,amssymb,amsthm}
\usepackage{listings}
\usepackage{diagbox}

\begin{document}

\author{Jason Bell}
\address[J. Bell]{Department of Pure Mathematics \\
University of Waterloo\\
Waterloo, ON  N2L 3G1 \\
Canada}
\email[J. Bell]{jpbell@uwaterloo.ca}
\thanks{Research of the first-named author supported by NSERC Grant RGPIN-2016-03632.}

\author{Haggai Liu}
\address[H. Liu]{Department of Mathematics\\
Simon Fraser University\\
Burnaby, BC  V5A 1S6 \\
Canada}
\email[H. Liu]{haggail@sfu.ca}

\author{Marni Mishna}
\address[M. Mishna]{Department of Mathematics\\
Simon Fraser University\\
Burnaby, BC  V5A 1S6 \\
Canada}
\email[M. Mishna]{mmishna@sfu.ca}
\thanks{Research of the third-named author supported by NSERC Grant RGPIN-2017-04157}

\title{Cogrowth series for free products of finite groups}

\begin{abstract}
  Given a finitely generated group with generating set~$S$, we study
  the \emph{cogrowth} sequence, which is the number of words of
  length $n$ over the alphabet $S$ that are equal to one.  This is
  related to the probability of return for walks the corresponding
  Cayley graph.  Muller and Schupp proved the generating function of
  the sequence is algebraic when $G$ has a finite-index free subgroup
  (using a result of Dunwoody). In this work we make this
  result effective for free products of finite groups: we determine bounds for the degree and height of
  the minimal polynomial of the generating function, and determine the
  minimal polynomial explicitly for some families of free products.  Using these results we
  are able to prove that a gap theorem holds: if~$S$ is a finite
  symmetric generating set for a group~$G$ and if~$a_n$ denotes the
  number of words of length $n$ over the alphabet~$S$ that are equal
  to~$1$ then $\limsup_n a_n^{1/n}$ exists and is either $1$, $2$, or at least $2\sqrt{2}$.
\end{abstract}

\maketitle
\tableofcontents

\section{Introduction}
\label{sec:Intro}
 Given a group~$G$ with finite generating set~$S$ the
cogrowth problem considers elements in the free monoid, $S^*$, on $S$ whose images in $G$ are
equal to the identity; this is a sublanguage of $S^*$. We denote the
set of these elements by 
$\mathcal{L}(G;S)$, or simply $\mathcal{L}$ when the context is clear. For many classes of groups the word problem is decidable
(i.e., there exists a decision procedure for determining if a word on
a set of generators is equal to the identity). This includes free
groups, one-relator groups, polycyclic groups and fundamental groups
of closed orientable two-manifolds of genus greater than one. On the
other hand, there exist groups with unsolvable word problem, with the
first such example being given by Novikov~\cite{novikov}. 


Here we consider an enumerative version of this problem, in which we are interested in
the number of words of a given length in $\mathcal{L}$. This is called
the \emph{cogrowth function} and we denote the number of words of length $n$
by  ${\rm CL}(n;G,S)$. The generating function of this counting
sequence,  $F_{G;S}(t)$, is called the
corresponding cogrowth series and is defined \begin{equation}
F(t)= F_{G;S}(t) :=
  \sum_{n\ge0} {\rm CL}(n;G,S)t^n.\end{equation}
  
This enumeration problem gives rise to three different notions of complexity. First, there is the complexity of the language $\mathcal{L}(G,S)$, which can be understood using the Chomsky-Sch\"utzenberger hierarchy of grammars. Next there is the complexity of the associated generating series, where we regard the class of rational power series as being simplest, with algebraic power series, $D$-finite series, and differentially algebraic series being viewed as increasingly complex classes of power series. Finally, there is the complexity of the group itself, which is typically understood in terms of desirable group theoretic properties.

When one works at the very simplest level, the three notions of complexity coincide: the language $\mathcal{L}(G,S)$ is a regular language if and only if $G$ is finite~\cite{anisimov} and this occurs if and only if $F_{G,S}(t)$ is rational.  Beyond this there are other well known connections: the generating function of an unambiguous context-free language satisfies a polynomial equation and hence is an
algebraic power series~\cite[Chapter III]{kuich}, and a result of Muller and Schupp \cite{muller} gives that this occurs precisely when $G$ has a finite-index free subgroup (i.e., $G$ is virtually free). 
Although, virtually
free groups are in some respect complex (e.g., they are non-amenable except in the cyclic-by-finite case), the word problem is relatively straightforward for this class of groups and thus in the sense of complexity of the word problem the class of virtually free groups is ``simple''. 

While there is a strong overlap between the various notions of
complexity (group theoretic, language theoretic, and complexity of
power series), there are nevertheless some families of groups which
are typically regarded as being structurally well behaved whose
corresponding cogrowth series are complex according to our notion of
complexity.  For example, it is shown in ~\cite{BeMi20} that a
finitely generated amenable group that is not virtually nilpotent can
never have a generating series that satisfies a non-trivial
homogeneous linear differential equation with rational function
coefficients.

In this work we focus on groups that are virtually free; specifically these are groups that have a finite-index free
subgroup. As remarked earlier, the cogrowth series are algebraic in this case, and in many cases it is useful to understand the polynomial equation that these generating functions
satisfy. 
\begin{figure}
	\begin{subfigure}{.33\textwidth}
	\centering
	\includegraphics[width=0.5\linewidth]{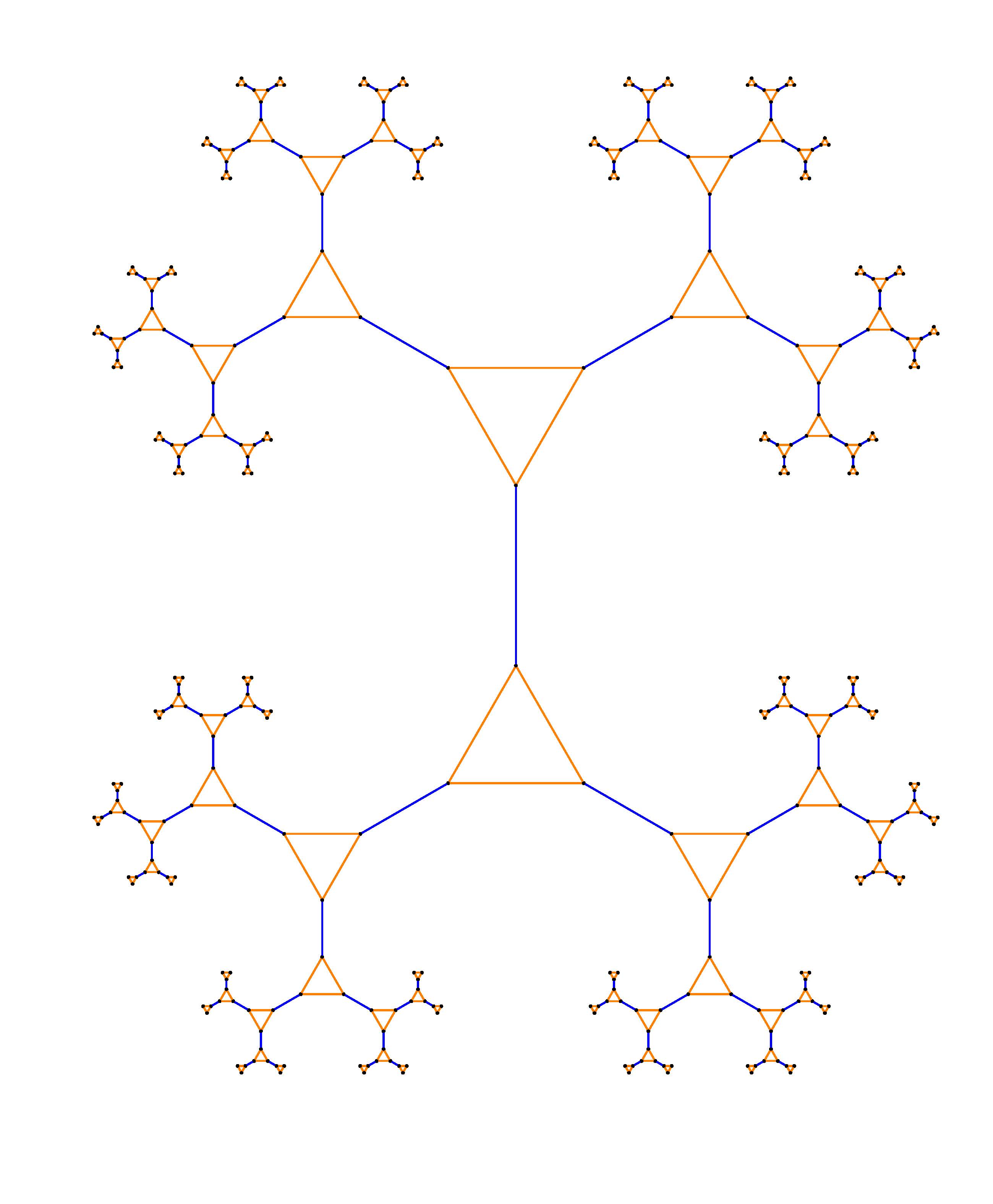}
	\caption{$m=2,n=3$}
	\end{subfigure}%
	\begin{subfigure}{.33\textwidth}
	\centering
	\includegraphics[width=0.7\linewidth]{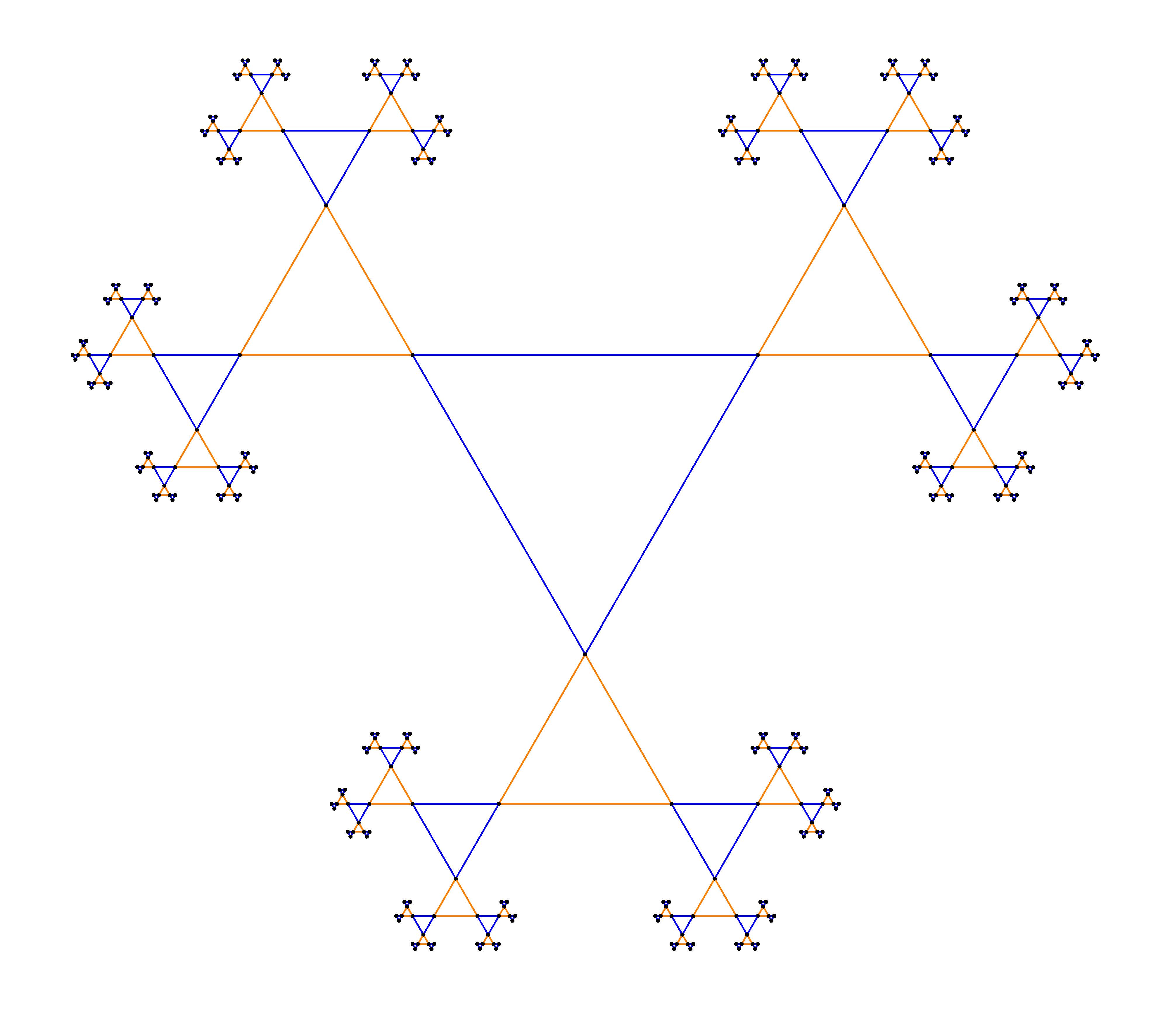}
	\caption{$m=n=3$}
	\end{subfigure}%
	\begin{subfigure}{.33\textwidth}
	\centering
	\includegraphics[width=0.63\linewidth]{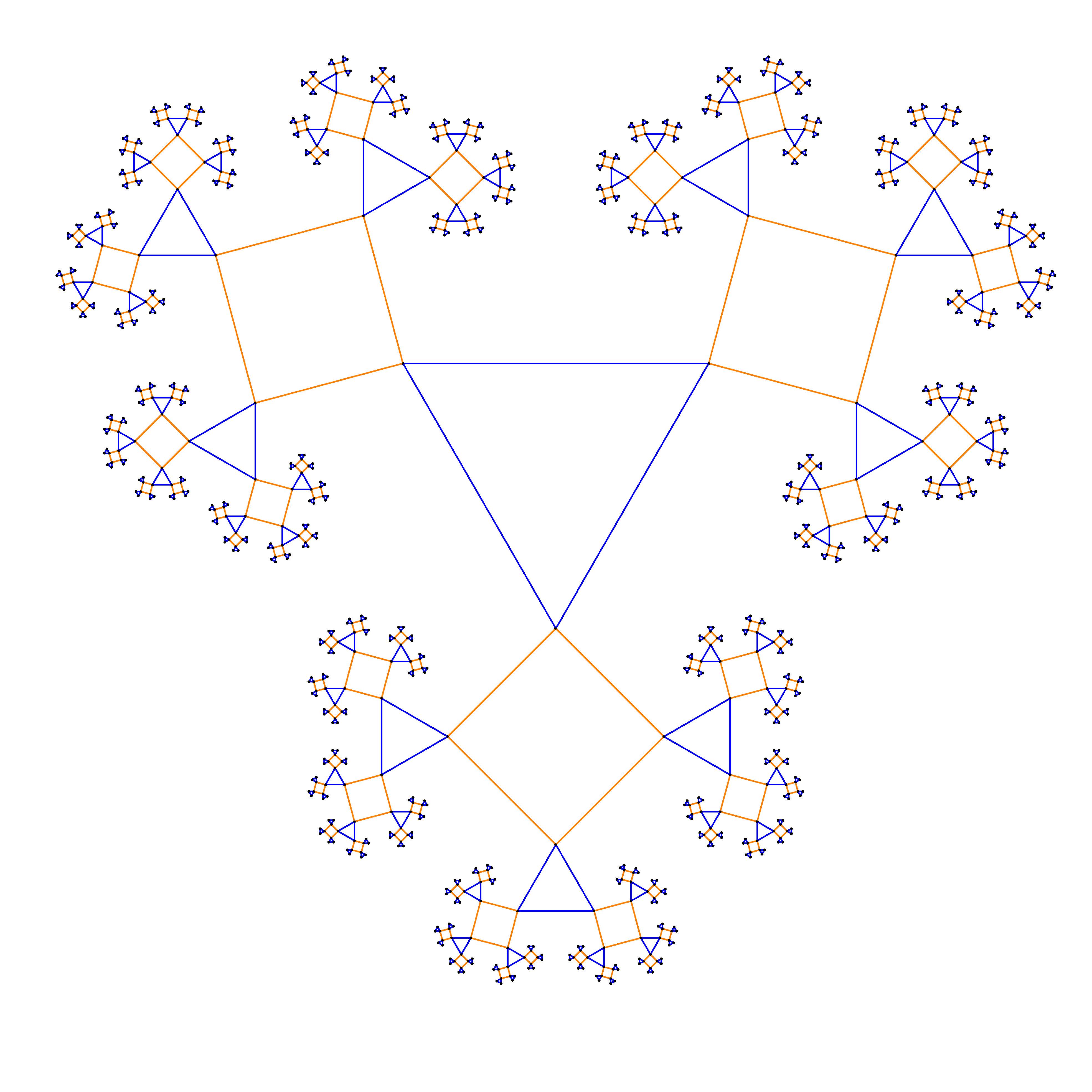}
	\caption{$m=3,n=4$}
	\end{subfigure}
	
	\vspace{10pt}
	\begin{subfigure}{.5\textwidth}
	\centering	
        \includegraphics[width=0.5\linewidth]{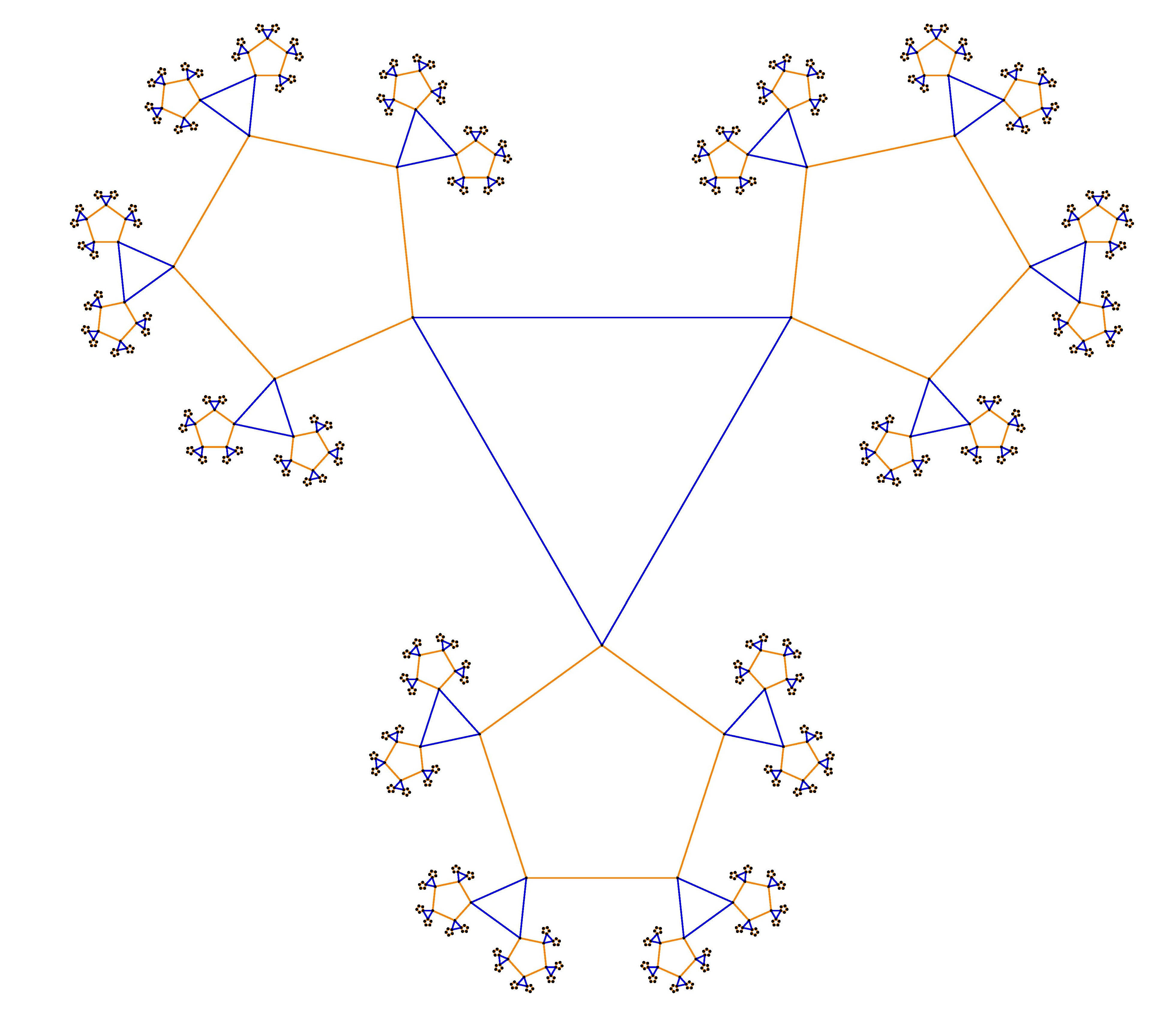}
	\caption{$m=3,n=5$}
	\end{subfigure}%
	\begin{subfigure}{.5\textwidth}
	\centering
	\includegraphics[width=0.495\linewidth]{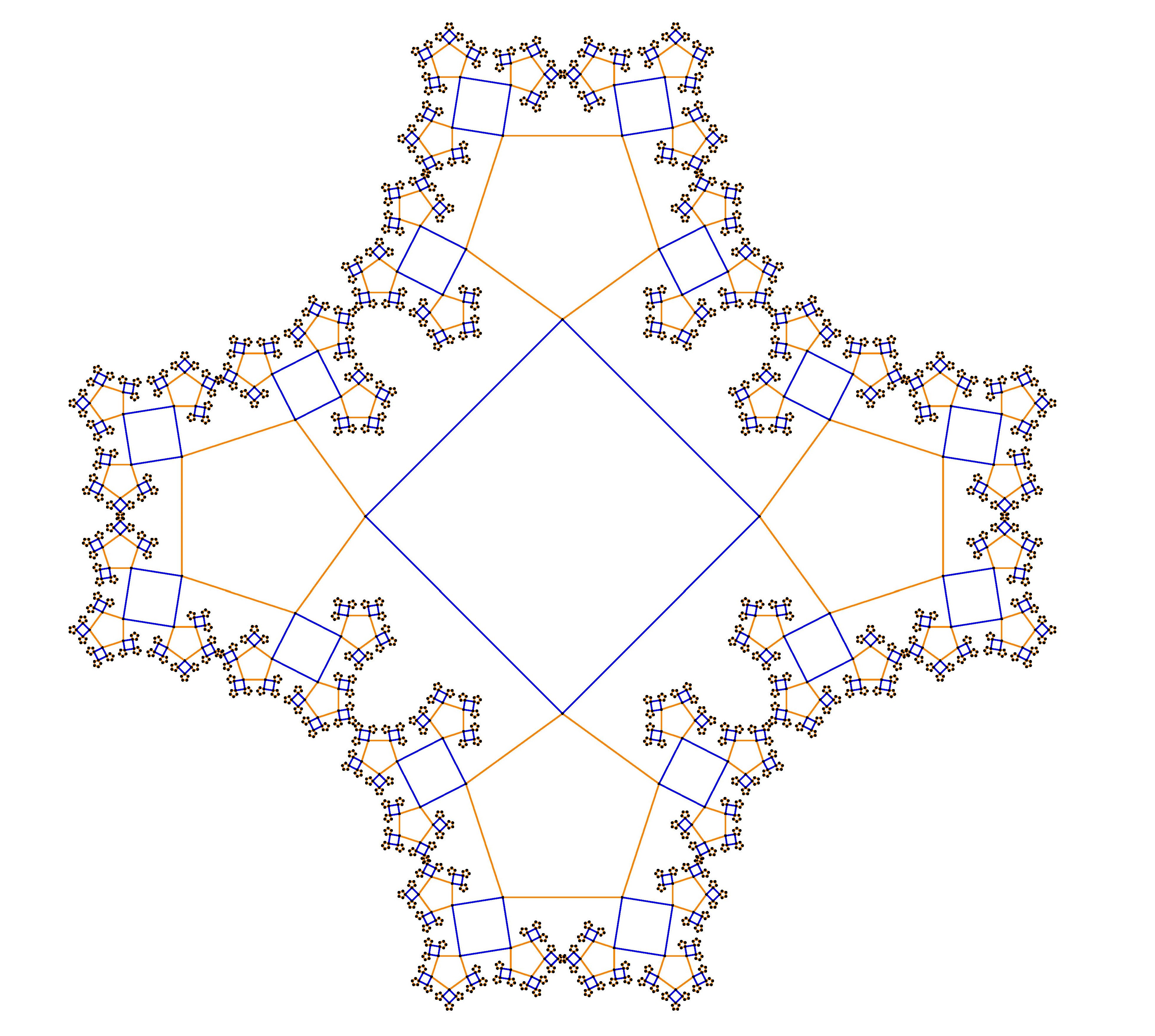}
	\caption{$m=4,n=5$}
	\end{subfigure}
\caption{\small Cayley graphs $X(G;S)$ of some virtually free groups. Each
  group is a free products of two cyclic groups
  $G=\Z/m\Z\star\Z/n\Z=\gen{x,y|x^m=1,y^n=1}$ with generating set
  $S=\{x,\inv{x},y,\inv{y}\}$.}
\label{fig_CG_2cylic}
\label{fig:graphs}
\end{figure}

Using a combinatorial argument, we are able to determine a grammar to
generate the language $\mathcal{L}(G;S)$ when $G$ is the free product
of a finite number of groups in which the groups are either finite or infinite cyclic. From the grammar it is
then straightforward to deduce a family of algebraic equations which yield
the cogrowth series. We use this system to find some explicit
expressions for series. We also outline a second strategy using free
probability to determine bounds~\cite{liu} on the degree and height of a
bivariate polynomial $\Lambda(t,z)\in \mathbb{Z}[t,z]$ such that
$\Lambda(t, F(t))=0$. Such bounds can be combined with strategic guessing
techniques in order to deterministically compute $\Lambda$.

The Cayley graphs of free products of finite groups provide a graph-theoretic
interpretation of this language. Recall,
that the Cayley graph of a group $G$ with given generating set $S$ is
the directed graph denoted $\mathcal{X}(G;S)$ with vertex set given by the elements of
$G$ with edge set $\{(g, gs): s\in S\}$. If the generating set is
closed under taking inverses, then we identify the edges $(g,gs)$ and $(gs,g)$ and the graph is considered to be undirected. The
words in $\mathcal{L}(G;S)$ are in bijection with walks on $\mathcal{X}(G;S)$
that start and end at the group identity element.
Figure~\ref{fig:graphs} illustrates subgraphs of Cayley graphs of some virtually
free groups, specifically neighbourhoods of the identity. To imagine
the full graph, recall that Cayley graphs are vertex transitive, and hence
these graphs are all infinite, and fractal in nature. 

The precise statements of the results that we can obtain are as
follows. We give explicit bounds on the degree and height of the minimal
polynomial of the cogrowth series for a group that is the free product
of a finite number of finite groups. This is proved in~\S\ref{sec:bounds}.
\begin{theorem}
\label{thm:mainbound}
   Let $G_1,\ldots ,G_r$ be finite groups with generating sets $S_1,\hdots, S_r$ respectively, let $G=G_1^{\star m_1}\star\dots G_r^{\star m_r}$, and let $S:=\bigcup_{i=1}^r \bigcup_{j=1}^{m_i} S_i^{(j)}\subseteq G$, where for each $i$, $S_i^{(1)},\ldots ,S_i^{(m_i)}$ are copies of $S_i$ in the corresponding copies of $G_i$ used in the formation of $G$.  Then the cogrowth series $F(t):=F_{G;S}(t)$ of $G$ with respect to $S$ is algebraic and satisfies $\Lambda(t,F(t))=0$, where $\Lambda(t,z)$ is a nonzero polynomial with rational coefficients with
 $${\rm deg}_t(\Lambda), {\rm deg}_z(\Lambda) \le \left(\prod_{i=1}^r \Delta_i\right) \left( 1 + \sum_{i=1}^r \frac{1}{\Delta_i} \right),$$
 where
   $\Delta_i$ is the sum of the degrees of the irreducible representations of $G_i$ for $i=1,\ldots ,r$.
In particular, the degrees do not depend on $m_1,\ldots ,m_r$ when we choose
 $S$ as above.
 \end{theorem}

In fact, we are able to compute the polynomials $\Lambda(t,z)$ in the statement of Theorem \ref{thm:mainbound} for several families of free products.
The results in Theorem~\ref{thm:main2} are worked out
in~\S\ref{sec:exam}.
\begin{theorem} 
\label{thm:main2}
We have the following results.
\begin{enumerate}[itemsep=15pt,label=(\alph*)]
\item Let $d,m \ge 2$. If $G=\langle x_1 ~|~x_1^d=1\rangle\star \cdots \star \langle x_m~|~x_m^d=1\rangle \cong \left(\mathbb{Z}/d\mathbb{Z}\right)^{\star m}$ and $S=\{x_1,\ldots ,x_m\}$, then $F(t)=F_{G,S}(t)$ is the unique solution to the equation
$$m^d t^d F(t)^d = (F(t)-1)(F(t)+m-1)^{d-1}\qquad {\rm with}~F(0)=1,$$ and $F(t)$ has radius of convergence  $(d-1)^{(d-1)/d}/(d(m-1)^{1/d})$.\label{example:a}

\item Let $m, s\ge 0$ be integers with $m+2s\ge 2$. If $G=\langle x_1 ~|~x_1^2=1\rangle\star \cdots \star \langle x_m~|~x_m^2=1\rangle \star \langle y_1,\ldots, y_s \rangle \cong \left(\mathbb{Z}/2\mathbb{Z}\right)^{\star m} \star \mathbb{Z}^{\star s}$ and  $S=\{x_1,\ldots ,x_m,y_1,y_1^{-1},\ldots ,y_s, y_s^{-1}\}$ then $F(t)=F_{G,S}(t)$ is the unique solution to the equation

$$(m+2s)^2 t^2 F(t)^2 = (F(t)-1)(F(t)+m+2s-1)\qquad {\rm with}~F(0)=1,$$ and $F(t)$ has radius of convergence  $1/(2\cdot (m+2s-1)^{1/2})$.
\label{example:b}

\item If $n\ge 2$ and $G=\langle x~|~x^2=1\rangle \star \langle y ~|~y^n=1\rangle \cong \mathbb{Z}/2\mathbb{Z}\star \mathbb{Z}/n\mathbb{Z}$ and $S=\{x,y\}$ then 
$$F_{G,S}(t)= (1-tD)/((1-tD)^2-t^2),$$ where $D$ is the unique power series solution to the equation
$$t^{n-1}(1-tD)^{n-1} = (1-tD-t^2)^{n-1} D$$ whose expansion begins $t^{n-1}+{\rm higher~degree~terms}$.
\label{example:d}

\end{enumerate}
\end{theorem}

We compute the power series $D$ in Theorem \ref{thm:main2} (b) for small values of $n$ in Example \ref{exam:3}. Additionally, as a consequence of Theorem~\ref{thm:main2}, we are able to prove a gap result.
\begin{theorem}\label{thm:main3}
Let $G$ be a finitely generated group with finite symmetric generating set $S$ and let $\rho_{G,S}$ denote the radius of convergence of the cogrowth generating series of $G$ with respect to $S$.  Then $\rho_{G,S}^{-1} \in \{1,2\}\cup [2\sqrt{2},\infty)$.
\end{theorem}
We suspect that all values in $[2\sqrt{2},\infty)$ can occur as the inverse of the radius of convergence of a cogrowth series, since the groups for which $2\sqrt{2}$ is realized as the inverse of a radius of convergence have uncountably many homomorphic images, although we have no evidence that all values in this interval can be realized; it is an interesting problem to determine the possible radii of convergence for a cogrowth generating series for a group with a symmetric generating set $S$.

Finite free products of finite groups and cyclic groups are virtually
free, and so there is a pushdown automaton that accepts the language of
words on~$S$ equal to the identity. In theory, one can translate the
automaton theoretic description to a give a description in terms of grammars, and one can then use this to find a
system of equations. Kuksov~\cite{kuksov} directly describes a
recursive system which he solves to find generating series for some of
the cases of Theorem~\ref{thm:main2} under the condition that one does
not allow ``doubling back'' on the Cayley graph; that is, one does not
allow a symbol $x$ to appear immediately next to the symbol $x^{-1}$
in the words considered. Nevertheless, it appears our systems resolve more cases
than were previously known.  Kuksov~\cite{kuksov} also obtained the
result of Theorem~\ref{thm:main2}\ref{example:a} in the case when
$d=2,3$. These two counting sequences appear the Online Encyclopedia
of Integer Sequences~\cite{oeis} as sequences A183135 and A265434.
Alkauskas~\cite{alkauskas} worked out Theorem~\ref{thm:main2}\ref{example:d}
in the case when $m=2$ and $n=3$, refining it to actually get the
cubic equation for the cogrowth series (this is of special
significance because this corresponds to the group ${\rm
  PSL}_2(\mathbb{Z})$).  Theorem~\ref{thm:main3} was noticed as a
curiosity coming from computations done while proving 
Theorem~\ref{thm:main2} and working out other examples.

The outline of this paper is as follows: In~\S\ref{sec:grammar}, we
give equations for computing the cogrowth of finite free groups and
cyclic groups.  We also prove the equations have a unique set of
solutions in power series with a given initial condition and that the
solutions are algebraic, although, as noted earlier, the algebraicity
follows from the Chomsky-Sch\"utzenberger Theorem~\cite{chomsky} and
the work of Muller and Schupp~\cite{muller}. In \S\ref{sec:exam} we
work out several examples, which are listed in the statement of
Theorem~\ref{thm:main2}; The main result of \S\ref{sec:bounds} is a
general bound on the degree of the minimial polynomial of co-growth series for
the free products of finite groups, which gives Theorem \ref{thm:mainbound} as a consequence. These results use ideas from free probability, which were suggested to two of the authors by one of the referees for the earlier paper \cite{BeMi20}.  In \S\ref{sec:gap} we prove the gap result for
radii of convergence given in Theorem~\ref{thm:main3}, using the results from the preceding sections. 

\section{A grammar construction and system of equations}
\label{sec:grammar}
Computing the cogrowth of free products of free groups has been done
in a number of cases \cite{alkauskas}, \cite{kuksov}, \cite{kuksov2}.
We note that Kuksov's work is the most general, but as we have already
remarked, he computes cogrowth using an altered definition.  In particular, he only counts \emph{reduced words} in the generating set $S$ that are equal to $1$; that is, if $x, x^{-1}\in S$ then he does not allow $x$ to immediately follow $x^{-1}$ or $x^{-1}$ to immediately follow $x$ in the words he considers.  We prove an analogue of his result that allows ``doubling back'' on the Cayley graph. We give an explicit algebraic system satisfied by the generating function. Alkauskas~\cite{alkauskas} does a computation for ${\rm PSL}_2(\mathbb{Z})$, which is a free product of a cyclic group of order $2$ and a cyclic group of order $3$, using a non-symmetric generating set of size $2$.  

We fix the following notation for use throughout this section.
We let $m$ be a positive integer and we let $G_1,\ldots , G_m$ be groups;
We let $S_i\subseteq G_i$ be a generating set for $G_i$ for $i=1,\ldots ,m$;
We let $S=\cup S_i\subseteq G_1\star \cdots \star G_m =: G$ be a generating set for the free product of $G_1,\ldots ,G_m$, where we identify $G_i$ with its image in the free product under the canonical inclusion when forming $S$.
 
For each $g\in G$, and $X\subseteq G$, we let $\mcl_{g, X}(S)$ (or
simply $\mcl_{g,X}$ if $S$ is understood) be the language of words
with the property that $s_1\cdots s_n \in \mathcal{L}_{g, X}(S)$ for a word $s_1\cdots s_n$ of length $n$ over the alphabet $S$, if $s_1\cdots s_n = g$
and for $1\le i<n$ we have $s_1\cdots s_i\not\in X$.  In this case, we
say that all proper prefixes of $s_1\cdots s_n$ \emph{avoid} $X$. We let
$F_{g,X}(t)$ be the ordinary generating function for the language $\mathcal{L}_{g,X}(S)$.

In the following lemma, we take $+$ to be combinatorial sum, i.e. disjoint
  union, and $\cdot$ to be concatenation, and for a statement $\mathrm{P}$ we take $\chi(\mathrm{P})$ to be $1$ if $\mathrm{P}$ is true and $0$ if $\mathrm{P}$ is false; in the case that $Z$ is a set, we take $\chi(\mathrm{P})Z$ to be empty if $\mathrm{P}$ is false; and the set $Z$ if $\mathrm{P}$ is true.  
  \begin{lemma}\label{lem:eq} Adopt the notation above.  Then for $i\in\{1,\ldots ,
  m\}$, each $g\in G_i$, and $X\subseteq G_i$, we have the following
  combinatorial relations:
  
\begin{align}
\label{lem:eq:case_1inX_gnot1}
\mcl _{g,X}&=\left(\chi\left(g\in S_i\cap X\right)\{g\} \right) 
                           + \bigcup_{s\in S_i\setminus X}
                   \left(\{s\}\cdot
                   \mcl _{s^{-1}g,s^{-1}X}\right)& 1\in X, g\neq~1;\\
\label{lem:eq:case_1notinX_gnot1}
\mcl _{g,X}&=\mcl _{1,X}\cdot \mcl _{g,\unone{X}}&1\notin X,\ g\neq 1;\\
\label{lem:eq:case_1notinX_g1}
\mcl _{1,X}&=\epsilon + \left(\mcl _{1,X}\cdot \left(\mcl _{1,\unone{X}}\setminus\epsilon\right)\right),&1\notin X;\\
\label{lem:eq:case_1inX_g1}
\mcl _{1,X}&=\epsilon + \bigcup_{s\in S\setminus
                   S_i}\left(\{s\}\cdot
                   \mcl _{s^{-1},\{s^{-1}\}}\right)+\bigcup_{s\in
                   S_i\setminus X}\left(\{s\}\cdot
                   \mcl _{s^{-1},s^{-1}X}\right)&1\in  X.
\end{align}
\end{lemma}
\begin{proof}
There are two main ideas in this proof. If $X$ does not contain 1,
then we can factor a walk based on its last passage through 1. If it
does contain 1, then the fact that the $G_i$ are disjoint gives a
restriction that allows us to decompose walks uniquely according to
their first step.  

Suppose first that $g\in G_i\setminus \{1\}$ for some
$i\in \{1,\ldots ,m\}$ and that $1\in X$. If it is that $s_1\cdots s_n=g$
and $s_1\in G_j\setminus \{1\}$ with $j\neq i$ then if
$s_1\cdots s_n=g$, the universal property for free
products gives that some prefix of $s_1\cdots s_n$
must be equal to $1$.
Hence every word of length $n$ starting with
$s_1\not\in G_i$ that is equal to $g$ must pass through $X$.  Thus if
$s_1\cdots s_n=g$ and every proper prefix avoids $X$ then $s_1$ must be in
$G_i$. If $n=1$, it may be
that $g\in X$ however,  if $n>1$ then $s_1\not\in X$, otherwise it
contradicts the prefix condition.   Then
$s_2\cdots s_n= s_1^{-1}g$ and every prefix of $s_2\cdots s_n$ avoids
$s_1^{-1}X$.  Moreover, if $s_2\cdots s_n= s_1^{-1}g$ and every prefix
of $s_2\cdots s_n$ avoids $s_1^{-1}X$ then appending $s_1$ at the
beginning gives a word of length $n$ that is equal to $g$ and such
that every prefix avoids $X$.  Hence we see that
\[\mathcal{L}_{g,X}= \chi(g\in S_i){g} + \bigcup_{s\in S_i\setminus X} \{s\}\cdot
\mathcal{L}_{s^{-1}g,s^{-1}X}.\]

Next, if $g\in G_i\setminus \{1\}$ for some $i\in \{1,\ldots ,s\}$ and $1\not\in X$
then we can factor the words in $\mathcal{L}_{g,X}$ uniquely by taking the largest prefix whose
product is equal to 1. In particular, if 
$s_1\cdots s_n\in \mathcal{L}_{g,X}$, we let $i<n$ denote the largest index such that $s_1\cdots s_i=1$.  (We note that $i=0$ is possible.) In this case we have a decomposition of $s_1\cdots s_n$ into a product $ab$ with $a=s_1\cdots s_i$ being equal to $1$ and such that every prefix of $a$ avoids $X$ and the word $b=s_{i+1}\cdots s_n$, which is equal to $g$ and such that every proper prefix avoids $X$ and also $1$.
Thus we get 
\[\mathcal{L}_{g,X} =  \mathcal{L}_{1,X}\cdot \mathcal{L}_{g,X\cup \{1\}}.\]

The cases when $g=1$ are similar, although we must 
account for the empty word, which, by convention, evaluates to 1. More precisely, if
$X\subseteq G_i$ and $1\in X$ then for $s\in S$, if we count words
$s_1\cdots s_n$ that are equal to $1$ and such that $s_1=s\not\in X$
and such that every prefix avoids $X$ then for $n\ge 1$ there is a
bijection between the collection of words of length $n$ and words
$s_2\cdots s_n$ of length $n-1$ that are equal to $s_1^{-1}$ and such
that every prefix avoids $s_1^{-1}X$.  Then
\[\mathcal{L}_{1,X}(t) = 1+\bigcup_{j=1}^m \bigcup_{s\in S_j\setminus
    X} \{s\} \cdot \mathcal{L}_{s^{-1}, s^{-1}X}.\]
\end{proof}

There is a classical translation of these combinatorial
equivalences to give a system of functional equations for the
corresponding ordinary generating functions.
    
\begin{corollary}\label{grammarCorrectness_series}
Adopt the notation above.  Then we have the following relations:
\begin{align}
\label{case_1inX_gnot1_series}
F_{g,X}(t)&=\chi(g\in S_i\cap X)t + \sum_{s\in S_i\setminus X}t
                                    F_{s^{-1}g,s^{-1}X}(t) &\text{ if }1\in X,\ g\neq 1;\\
\label{case_1notinX_gnot1_series}
F_{g,X}(t)&=F_{1,X}(t)F_{g,\unone{X}}(t) &\text{ if }1\notin X,\ g\neq 1;\\
\label{case_1notinX_g1_series}
F_{1,X}(t)&=1+F_{1,X}(t)(F_{1,\unone{X}}(t)-1)&\text{ if }1\notin X;\\
\label{case_1inX_g1_series}
F_{1,X}(t)&=1+\sum_{s\in S\setminus
            S_i}tF_{s^{-1},\{s^{-1}\}}(t)+\sum_{s\in S_i\setminus X} tF_{s^{-1},s^{-1}X}(t)&\text{ if  }1\in X.
\end{align}
\end{corollary}
\begin{proof} This is done via a well known translation, as in~\cite{flajolet}. \end{proof}

\begin{remark}
Although the system in Lemma~\ref{lem:eq} is not \emph{a priori} finite when the groups are not finite, one can easily adapt this construction to handle the case where some of the $G_i$ are infinite cyclic groups with $S_i=\{x,x^{-1}\}$ with $x$ a generator for $G_i$.  The reason for this is that if a word $s_1\cdots s_n$ has some proper prefix equal to $x^i$ with $i>0$ then it has a proper prefix equal to $x^j$ for $0<j<i$; similarly, if it has a proper prefix equal to $x^i$ with $i<0$ then it has a proper prefix equal to $x^j$ for $j<0$ with $j>i$.  Thus we only need to consider $X$ with at most three elements (potentially one positive exponent, one negative exponent, and the identity), since if $x^i$ and $x^j$ are in $X$ with $i>j>0$ then $F_{g,X}=F_{g,X\setminus \{x^i\}}$ and an analogous result holds for negative exponents.  Also, if $i>j>0$ then and $x^j\in X$ then $F_{x^i,X}=0$ and similarly in the negative case.  Thus if one uses these facts and looks at the dependency tree that arises when looking at equations  from Lemma, then we see that in the case that $G_i$ is an infinite cyclic group with generator $x$, we only need to consider $F_{g,X}$ for $g\in G_i$ and $X\subseteq G$ with $g\in \{x^{-1},1,x\}$ and $X\subseteq \{x^{-1},1,x\}$.  See Example~\ref{exam:2} for a case where this is implemented.
\label{rem:infinite}
\end{remark}

\subsection{Examples}
\label{sec:exam}
We now give some examples in which we use the systems above to compute cogrowth series for certain free products.  For the following examples, we generated the system given by the
equations, and applied simplifications. The solvable examples
generally possess an exploitable symmetry, so although in theory one
might have to manipulate $\sum_{i=1}^m  |G_i|2^{|G_i|}$ equations, in
practice there are far fewer.  We incrementally eliminated the
variables\footnote{Specifically, we have used the {\tt eliminate} command of
  Maple 2018.} to determine the algebraic equation satisfied by
$F_{1,\emptyset}$. When listed, the OEIS numbers refer to the Online
Encyclopedia of Integer
Sequences~\cite{oeis}. Table~\ref{tab:examples} in the Appendix summarizes the sequences.

We start with the following infinite family of groups.
\begin{example}
\label{exam:1}
Let $d, m\ge 2$ and let $G=\left(\mathbb{Z}/d\mathbb{Z}\right)^{\star
  m}=\langle x_1~ |~x_1^d=1\rangle \star \cdots \star \langle x_m
~|~x_m^d=1\rangle$ and let $S=\{x_1,\ldots ,x_m\}$.  Then the
generating series, $Z(t)$ for ${\rm CL}(n;G,S)$, is the unique power
series satisfying the equation
\[m^d t^d Z^d = (Z-1)(Z+m-1)^{d-1}\] 
with initial condition $Z(0)=1$.
\end{example}

Below we prove this using the grammar, however in
Section~\ref{sec:AlgSys} we give a second proof as a consequence of a more general result obtained by methods from free probability. 
\begin{proof}
It is straightforward to solve the system from Lemma~\ref{lem:eq} with
an additional remark. We let
\begin{equation}
A:=F_{1,\{x_1\}}(t), B=F_{1,\{1,x_1\}}(t), C=F_{x_1^{-1},\{x_1^{-1}\}}(t).
\end{equation}
Notice that a word $s_1\cdots s_n\in S^n$ that is equal to $x_1^{-1}$ and has no proper prefix equal to $x_1^{-1}$ can be decomposed uniquely in the form
$w_1 x_1 w_2 x_1 \cdots w_{d-1} x_1$, where each $w_i$ is a word in $s_1,\ldots ,s_m$ that is equal to $1$ with no proper prefix equal to $x_1$.  This then gives us the equation
\begin{equation}
\label{eq:A}
t^{d-1} A^{d-1} = C.
\end{equation}
Equation~\eqref{lem:eq:case_1notinX_g1} of Lemma~\ref{lem:eq} gives that 
$A=1+ A(B-1)$; that is, $A=1/(2-B)$.  Finally, by symmetry we can rewrite Equation~\eqref{lem:eq:case_1inX_g1} of Lemma~\ref{lem:eq} as
\begin{equation}
\label{eq:B}
B=1+(m-1)t C.\end{equation}
Using the fact that $A=1/(2-B)$ then gives that
$A=(1-(m-1)tC)^{-1}$ and so Equation (\ref{eq:A}) yields
\begin{equation}
\label{eq:CC}
t^{d-1} = C(1-(m-1)tC)^{d-1}.\end{equation}

Let $W=F_{1,\{1\}}(t)$ and $Z=F_{1,\emptyset}(t)$.  Then $Z$ is the generating series for the cogrowth.
Equation~\eqref{lem:eq:case_1notinX_g1} of Lemma~\ref{lem:eq} gives 
\begin{equation} \label{eq:Z}Z=1+Z(W-1)\end{equation} and Equation~\eqref{lem:eq:case_1inX_g1} of Lemma~\ref{lem:eq}, again using symmetry, gives
\begin{equation}\label{eq:W}
W=1+mt C.
\end{equation}
Using Equations (\ref{eq:Z}) and (\ref{eq:W}), we see $Z=(1-mtC)^{-1}$, so substituting $C=(Z-1)/(mtZ)$ into Equation~\ref{eq:CC} gives
$$mt^{d}Z = (Z-1)(1-(m-1)(Z-1)/mZ)^{d-1},$$ or equivalently
$$m^d t^d Z^d = (Z-1)(Z+m-1)^{d-1},$$ as claimed.  To see uniqueness of the solution once we impose the initial condition $Z(0)=1$, note that if there is a unique polynomial solution of degree $n-1$ to this equation mod $(t^n)$ for $n\ge 1$ then we get a unique polynomial solution of degree $n$ to this equation mod $(t^{n+1})$ by Hensel's lemma and so by induction there is a unique power series solution with this initial condition.
\end{proof}
We discovered the equation in Example \ref{exam:1} by solving the system on Maple for $d$ with $d\le 9$, and symbolic $m$, and we were able to use the Maple package {\tt gfun}~\cite{gfun} to guess the general form of the algebraic equation satisfied by the cogrowth.  This then suggested a method for proving this fact. We are also able to determine dominant singularity of the cogrowth generating functions appearing in Example \ref{exam:1}. 
\begin{lemma}
Let $\beta>0$ and let 
\[P(z)=m^d \beta^d z^d - (z-1)(z+m-1)^{d-1}.\] Then $P(z)$ has a repeated root if and only if $\beta= (d-1)^{(d-1)/d}/(d(m-1)^{1/d})$.
\end{lemma}
\begin{proof}
Suppose that $P(z)$ has a repeated root.  Then the system $P(z)=P'(z)=0$ has at least one solution.  Notice that
$$P'(z)=0 \iff  d m^d \beta^d z^{d-1} = (z+m-1)^{d-1} + (d-1)(z-1) (z+m-1)^{d-2}.$$
It is easy to see that if $P(z)=0$ then $z\neq 0,1,-m+1$.  If $P(z)=0$ then we have
$m^d \beta^d z^d=(z-1)(z+m-1)^{d-1}$, so
dividing the equation
$$d m^d \beta^d z^{d-1} = (z+m-1)^{d-1} + (d-1)(z-1) (z+m-1)^{d-2}$$ on the left by
$m^d\beta^d z^d$ and on the right by $(z-1)(z+m-1)^{d-1}$, we see that if $P(z)$ and $P'(z)$ are both equal to zero, then we must have
$$d/z = 1/(z-1) + (d-1)/(z+m-1).$$  Multiplying by $z(z-1)(z+m-1)$ then gives the equation
$$d(z-1)(z+m-1) = z(z+m-1) + (d-1) z(z-1),$$ which has the unique solution
$z=z_0:=(m-1)d/(dm-d-m)$.

Thus if $P(z)=P'(z)=0$ then we must have $z=z_0$ and so we substitute $z=z_0$ into $P(z)$ and use the fact that $P(z_0)$ must be zero to get
$$m^d \beta^d (m-1)^d/(d-2)^d - (m-d+1)/(d-2) \cdot (m-1)^{d-1}(d-1)^{d-1}/(d-2)^{d-1},$$
which has the solution
$$\beta^d =(d-1)^{d-1}/((m-1)d^d),$$ which has the unique positive solution
$$\beta= (d-1)^{(d-1)/d}/(d(m-1)^{1/d}).$$ Thus we have shown that $P(z)$ has a double root only when $\beta= (d-1)^{(d-1)/d}/(d(m-1)^{1/d})$, and this double root occurs at 
$$z=(m-1)d/(dm-d-m).$$ We also get that $P'(z_0)=P(z_0)=0$ for this value of $\beta$, and so the result follows.
\end{proof}
\begin{corollary}\label{cor:main3} Let $m,d\ge 2$.  The radius of convergence of the cogrowth generating function for  
$$G=\langle x_1~|~x_1^d=1\rangle \star \cdots \star \langle x_m~|~x_m^d=1\rangle\simeq (\Z/d\Z)^{\star m}$$ 
with respect to $S=\{x_1,\ldots ,x_m\}$ is 
$$\displaystyle{{(d-1)^{\frac{d-1}{d}}\over d(m-1)^{\frac{1}{d}}}}.$$

\end{corollary}
\begin{proof}
The singularities of an algebraic power series $F(t)$ satisfying a polynomial
equation\\ $\Lambda(t,F(t))=0$ for some polynomial $\Lambda(t,z)\in \mathbb{C}[t,z]$ are in the set $T$, where $T$ is the set of zeros of the leading coefficient of $\Lambda(t,z)$ as a polynomial in $z$ and the zeros of the discriminant of $\Lambda(t,z)$ with respect to $z$ (see Flajolet and Sedgewick \cite[\S7.36]{flajolet}). In the case that $F(t)$ is the cogrowth generating function of $G$ with respect to $S$, we have $F(t)$ is a root of $\Lambda(t,z)$, where \[\Lambda(t,z)=m^dt^{d}z^d -(z-1)(z+m-1)^{d-1},\] which has leading coefficient $m^d t^d - 1$.  We showed that $\Lambda(t,z)$ can only have repeated roots for $t\ge 0$ at 
$t=  (d-1)^{(d-1)/d}/(d(m-1)^{1/d})$.  Thus the only positive singularities of $F(t)$ are in 
$T\cap (0,\infty)= \{(d-1)^{(d-1)/d}/(d(m-1)^{1/d}),1/m\}$.  Since $F(t)$ has nonnegative coefficients, it has a singularity at $t=\rho$, where $\rho>0$ is the radius of convergence.  If $d,m\ge 2$ and $(d,m)\neq (2,2)$, then $G$ is nonamenable, and a strengthening of Kesten's criterion due to Gray and Kambites~\cite[Corollary 6.6]{Kam} gives that the radius of convergence is strictly less than $1/|S|=1/m$.  Thus the radius of convergence is $(d-1)^{(d-1)/d}/(d(m-1)^{1/d})$ for $(d,m)\neq (2,2)$. When $d=m=2$, $G$ is amenable and $(d-1)^{(d-1)/d}/(d(m-1)^{1/d})=1/m$, so the result follows.
\end{proof}
\begin{remark} Kuksov~\cite{kuksov} did this computation when $d=2$ and $d=3$, but did not do the general case.  The case $d=2$ is classical, as it can be interpreted in terms of rooted closed walks of length $2n$ on the infinite rooted $m$-ary tree.  The cases $d=2$ and $d=3$ appear in the OEIS as entries A126869 
 and A265434, respectively.
\end{remark}
\begin{example} \label{exam:2} Let $G = \langle x~|~x^2=1\rangle \star \langle y\rangle$ and let $S=\{x,y,y^{-1}\}$.  Then the cogrowth series for $G$ with respect to $S$ is equal to $$Z=\frac{1}{2}\cdot \left(1-3\sqrt{1-8t^2}\right)/(1-9t^2).$$
\end{example}
\begin{proof} Using the fact that $F_{y,Y}=F_{y^{-1},Y^{-1}}$, which follows from the obvious symmetry and using the equations in  Lemma~\ref{lem:eq} along with Remark \ref{rem:infinite} about how to apply them in the infinite cyclic case, we get the equations:
\begin{enumerate}
\item $F_{x,\{1,x\}}(t)=t$; 
\item $F_{x,\{x\}}(t)= t F_{1,\{x\}}(t)$;
\item $F_{1,\{x\}}(t) = 1 + F_{1,\{x\}}(t) (F_{1,\{1,x\}}(t) -1)$;
\item $F_{1,\{1,x\}}(t) = 1+ 2tF_{y,\{y\}}(t)$;
\item $F_{1,\{y\}}(t) = 1+F_{1,\{y\}}(t)(F_{1,\{1,y\}}(t)-1)$;
\item $F_{y,\{y\}}(t)=F_{1,\{y\}}(t) F_{y,\{1,y\}}(t)$;
\item $F_{y,\{1,y\}}(t)  = t$;
\item $F_{1,\{1,y\}}(t) = 1 + tF_{x,\{x\}}(t) + t F_{y, \{y\}}(t)$;
\item $F_{1,\{1\}}(t)=1+tF_{x,\{x\}}+2tF_{y,\{y\}}(t)$;
\item $F_{1,\emptyset}(t)=1 + F_{1,\emptyset}(t)(F_{1,\{1\}}(t)-1)$.
\end{enumerate}
  \end{proof}
Then we solve this equation using Maple and find that $Z=F_{1,\emptyset}(t)$ satisfies the polynomial equation
$$(3t-1)(3t+1)(t-1)(t+1)Z^3+(-10t^2+2)Z^2+(2t^2-1)Z-2=0,$$
which factors as
$$((9t^2-1) Z^2 -Z+2)((t^2-1)Z-1)=0.$$
Now $Z$ is a power series whose initial terms are $1+3t^2+\cdots$ and so we see that it is a root of the first factor, which we can solve: \[Z=\frac{1}{2}\cdot \left(1-3\sqrt{1-8t^2}\right)/(1-9t^2).\] The dominant singularity comes from the branch cut, and the radius of convergence is $1/(2\sqrt{2})$.  
\begin{remark}\label{remark:cogrowth_equiv} 
We note that the cogrowth series for $\mathbb{Z}/2\mathbb{Z}\star \mathbb{Z}$ given above is the same as the series for $d=2$, $m=3$ in Example \ref{exam:1}; that is, for the free product of three copies of the cyclic group of order $2$.  The reason for this can be seen by the fact that if $u,v$ are elements of order two that generate an infinite dihedral group and if $x,x^{-1}$ generate an infinite cyclic group, then if we let $f_i(u)=x$, $f_i(v)=x^{-1}$ for $i$ odd and let $f_i(u)=x^{-1}$ and $f_i(u)=x$ for $i$ even then we have a map $g:\{u,v\}^*\to \{x,x^{-1}\}^*$ given by
$g(a_1\cdots a_n)=f_1(a_1)f_2(a_2)\cdots f_n(a_n)$ for $a_1,\ldots
,a_n\in\{u,v\}$ and a straightforward induction shows that this map
gives that these two groups have the same cogrowth.

We note that this bijective argument holds more generally: the group $\mathbb{Z}\star H$
  and the group $(\mathbb{Z}/2\mathbb{Z})^{\star 2} \star H$ have the same cogrowth, for any finitely
  generated $H$: Suppose $\{x,x^{-1}\}$ are the generators for
  $\mathbb{Z}$, $u$ and $v$ are order 2 elements generating
  $\mathbb{Z}/2\mathbb{Z}$ and $T$ generates $H$, then if we have a
  word on $\{u,v\} \cup T$, $1= \mu_1 \tau_1 \mu_2 \tau_2 \dots
  \mu_r \tau_r$, with the $\tau_i$ a word on $T$ and $\mu_i$ a word in
  the $u$, $v$. We can define $f_i$ are above, and apply it to the the associated
  word $\mu_1\mu_2\cdots \mu_r$ to get $\mu_1' \mu_2'\cdots \mu_r'$ in $x$ and
  $x^{-1}$. We interlace this
  with the $\tau_i$ to create $\mu'_1 \tau_1 \mu'_2 \tau_2 \dots
  \mu'_r \tau_r$, the image of the word in $(\mathbb{Z}/2\mathbb{Z})^{\star 2} \star H$. 
\end{remark}

\begin{example} \label{exam:3}
Let $G=\mathbb{Z}/2\mathbb{Z}\star \mathbb{Z}/n\mathbb{Z}=\langle x ~|~x^2=1\rangle \star \langle y ~|~y^n=1\rangle$ and let $S=\{x,y\}$ and let $F(t)$ denote the cogrowth generating series for $G$ with respect to $S$.  Then $$F(t)= (1-tD)/((1-tD)^2-t^2),$$ where $D$ is the unique power series solution to the equation
$$t^{n-1}(1-tD)^{n-1} = (1-tD-t^2)^{n-1} D$$ whose expansion begins $t^{n-1}+{\rm higher~degree~terms}$.
\end{example}
\begin{proof}
We let
\begin{equation}
A_1:=F_{1,\{x\}}(t), B_1=F_{1,\{1,x\}}(t), C=F_{x^{-1},\{x^{-1}\}}(t)
\end{equation}
and
\begin{equation}
A_2:=F_{1,\{y\}}(t), B_2=F_{1,\{1,y\}}(t), D=F_{y^{-1},\{y^{-1}\}}(t).
\end{equation}

Notice that a word $s_1\cdots s_n\in S^n$ that is equal to $x^{-1}$ and has no proper prefix equal to $x^{-1}$ can be decomposed uniquely in the form
$w_1 x$, where each $w_i$ is a word in $s_1,\ldots ,s_m$ that is equal to $1$ with no proper prefix equal to $x$.  This then gives us the equation
\begin{equation}
\label{eq:AA}
t A_1 = C.
\end{equation}
Similarly, we have
\begin{equation}
\label{eq:AAA}
t^{n-1} A_2^{n-1} = D.
\end{equation}
We now make use of Lemma \ref{lem:eq} \eqref{lem:eq:case_1inX_gnot1}--\eqref{lem:eq:case_1inX_g1}.
Equation~\eqref{lem:eq:case_1notinX_g1} gives us that 
$A_1=1/(2-B_1)$ and $A_2=1/(2-B_2)$.  Equation~\eqref{lem:eq:case_1inX_g1} gives
\begin{equation}
\label{eq:BBB}
B_1=1+t D\qquad {\rm and}\qquad B_2=1+tC.\end{equation}
Combining these equations then gives that
$t =(1-tD) C$ and we have
$t^{n-1} = (1-tC)^{n-1} D$.
Thus we see that D is a solution to the equation
$$t^{n-1}(1-tD)^{n-1} = (1-tD-t^2)^{n-1} D.$$

Finally, Equation~\eqref{lem:eq:case_1inX_g1} gives $F_{1,\{1\}}(t) = 1+ t (C+D)$ and Equation~\eqref{lem:eq:case_1notinX_g1} gives $F(t)=F_{1,\emptyset}(t)=1/(1-tC-tD)$,
and so if we use the fact that $C=t/(1-tD)$, we see that
$$F(t)=(1-tD)/((1-tD)^2-t^2) = \frac{1}{2}\left(1/(1-t-tD) + 1/(1+t-tD)\right).$$
Uniqueness of $D$ after imposing the initial conditions follows from a standard application of Hensel's lemma.

For $n=3,4,5$ we get the following expressions for the minimal polynomial of $F(t)$ using Maple.  (Note that the case $n=2$ is done is the case $m=d=2$ in Example \ref{exam:1}.)
\begin{description}
\item[$n=3$]
\[((t-1)^3+t^3)((t+1)^3-t^3)Z^3+(t^5-t^4+t^3+2t^2-1)Z^2+(t^3-t^2+1)Z+1\]
\item[$n=4$]
\[((t^4-(t-1)^4)((t+1)^4-t^4)Z^4+2(4t^6-2t^4+3t^2-1)Z^3+t^4(t^2+3)Z^2+(t^4-2t^2+2)Z+1\]
\item[$n=5$]
\[((t-1)^5+t^5)((t+1)^5-t^5)\,Z^5
+A\,Z^4 +B\,Z^3+C\,Z^2
+Z+1.\]
\end{description}
Here
\[A=3(t^9-t^8+6t^7+4t^6+t^5-6t^4+4t^2-1)\qquad B=2(4t^7+t^6+3t^5-3t^4+3t^2-1)\]
\[C=(t^7+4t^5+2t^4-4t^2+2) \qquad D=(t^5-3t^2+3).\]

We computed the minimal polynomials for some higher $n$ too, but the expressions became increasingly unwieldy and we could not discern any obvious patterns governing the coefficients of the annihilating polynomials. One exception is the leading term which is predicted to be:

\begin{equation}
-\left((t+1)^n-t^n\right)\left((1-t)^n-t^n\right) Z^n. \end{equation}

\end{proof}
The case when $n=3$ was previously worked out by Alkauskas~\cite{alkauskas} and our formula appears in his Theorem 1. It corresponds to the cogrowth of ${\rm PSL}_2(\mathbb{Z})$ as a semigroup generated by two elements, one of order $2$ and another of order $3$.  Again, we apply standard techniques to this algebraic equation to deduce that the singularities of $F$ are contained in the set of zeros of the leading coefficient and the discriminant. For the radius of convergence, we are only interested in real singularities in the range $[1/2,\infty)$. In the case of zeros of the leading coefficient in the interval $[1/2,\infty)$, we need $((t-1)^3+t^3)=0$; that is, $t=1/2$. In the case of zeros of the discriminant, we compute the discriminant and find that it is
{\small
\[ \left( t^{13}-8t^{12}-4t^{11}+164t^{10}-392t^{9}+ 404{t}^{8}-752{t}^{7}+260{t}^{6}-512{t}^{5}-128{t}^{4}-160t^{3}-64{t}^{2}+64 \right) {t}^{3},\]}
which has a unique solution in $[1/2,\infty)$ that we numerically estimate to be~$.5072330945\ldots$. The radius of convergence is one of these two values. Since the cogrowth function is bounded above by $2^n$ and since $F(t)$ has a singularity at its radius of convergence by Pringsheim's theorem.  We again invoke~\cite[Corollary 6.6]{Kam} to get that the radius of convergence is strictly greater than $1/2$, since ${\rm PSL}_2(\mathbb{Z})$ is non-amenable, and so the radius of convergence is $0.50723\cdots$. For similar reasons, for general $n$ we predict the radius of convergence will be a zero of the discriminant, and not of the leading coefficient. 

We can now prove Theorem \ref{thm:main2}
\begin{proof}[Proof of Theorem \ref{thm:main2}] Part (a) follows from Example \ref{exam:1} and Corollary \ref{cor:main3}.  For part (b), observe that Remark \ref{remark:cogrowth_equiv} shows that 
$\left(\mathbb{Z}/2\mathbb{Z}\right)^{\star m} \star \left(\mathbb{Z}\right)^{\star s}$ has the same cogrowth as $\left(\mathbb{Z}/2\mathbb{Z}\right)^{\star(m+2s)}$, and so the result follows from (a). Part (c) follows from Example \ref{exam:3}.
\end{proof}

As motivated by Remark \ref{remark:cogrowth_equiv}, it would be interesting to know whether one can find all pairs of non-isomorphic groups with symmetric generating sets whose cogrowth series are the same.

\section{Bounding the degree of the minimal polynomial}
\label{sec:bounds}
In this section, we seek upper bounds for the degree of the minimal
polynomials of various cogrowth generating functions. With knowledge
of such upper bounds, along with sufficiently many terms in the series,
we can apply Pade approximant techniques~\cite{aecf} and undetermined
coefficient methods to obtain the precise minimal polynomial.

\subsection{A bound for free products of finite groups}
\label{sec_gen_finite_gps}
Finite groups have rational cogrowth generating functions, for which
the degree of the numerator and denominator can be effectively bounded
in terms of the degrees of their irreducible representations.  Since the free product of groups is the coproduct in the category of groups, we will use $\coprod_i G_i$ to denote the free product of groups.  We will also use 
$G^{\star m}$ to denote $\coprod_{i=1}^m G$.

We first set up some notation. Let $G_1,\ldots, G_r$ be finite groups with generating sets $S_1,\ldots ,S_r$ respectively, and let $m_1,\ldots ,m_r$ be positive integers. Consider the free product, 
$$G:=\coprod_{i=1}^r G_i^{\star m_i}.$$ 
For $i=1,\hdots,r$, let $S_i^{(1)},\ldots , S_i^{(m_i)}$ denote copies of $S_i$ in the respective copies of $G_i$ used in the formation of $G$. We take
\begin{equation}
S:=\bigcup_{i=1}^r \bigcup_{j=1}^{m_i} S_i^{(j)}\end{equation}

to be our generating set for $G$. We now adopt some notation from free probability. For each $i$, let $\alpha_i = \sum_{s\in S_i} s \in \mathbb{C}[G_i]$. For the group $G$ defined above, we let $\phi: \mathbb{C}[G]\to \mathbb{C}$ be the \emph{expectation operator}: the map that extracts the coefficient of $1_G$ from an element of the group algebra. Furthermore, for each $\beta\in \C[G]$, we consider its Cauchy Transform 
$$\CT_{\beta}(t):=\sum_{n\ge 0}\phi(\beta^n)t^{-n-1}\in \mathbb{C}[[t^{-1}]].$$ 
Observe that each element of $\C[G_i]$ can be canonically identified with an element of $\C[G]$ using any one of the $m_i$ copies of $G_i$ in $G$. Moreover, the value of the operator $\phi$ applied to an element of the group algebra $\C[G_i]$ is independent of the choice of embedding into $G$.  Hence, we may define $\CT_{\rho}(t)$ for $\rho\in \C[G_i]$ without any confusion. We let $\invCT_{\beta}(t)$ denote the compositional inverse of $\CT_{\beta}(t)$. Since we are working with power series in $t^{-1}$ this is saying
\begin{equation}\label{eq:inv}
\CT_{\beta}\circ \invCT_{\beta}(t) = \invCT_{\beta}\circ \CT_{\beta}(t) = t^{-1}.
\end{equation}

Finally, let $$\alpha = \sum_{i=1}^r\sum_{j=1}^{m_i} \sum_{s\in S_i^{(j)}} s\in \mathbb{C}[G].$$
Notice that by our choice of $\alpha$, 
\begin{equation}\label{eq:11}
\CT_{\alpha}(t)=\inv{t}F_{G;S}(\inv{t}).
\end{equation}
 Recall that each 
$F_{G_i;S_i}(t)$ is a rational function with integer coefficients. We
can deduce further results regarding the degrees of the numerator and the denominator. \begin{lemma}\label{lem:bound}
Let $H$ be a finite group with $d$ (inequivalent) irreducible representations with degrees given by $n_1,\ldots ,n_d$ respectively, let $T$ be a generating set for $H$, and let $F(t):=F_{H;T}(t)=\sum_{n\ge 0} \phi(\alpha^n) t^n$, where $\alpha=\sum_{s\in T} s\in \mathbb{C}[H]$, and $\phi$ is the expectation operator.
Then $F(t)$ is the power series expansion of a rational function $P(t)/Q(t)$ where $P,Q\in \mathbb{Z}[t]$ are polynomials with $Q(0)=1$ and the degrees of $P$ and $Q$ are bounded by respectively $n_1+\cdots + n_d-1$ and $n_1+\cdots + n_d$.
In particular, the degrees of $P$ and $Q$ are at most respectively $|H|-1$ and $|H|$ and equality can only occur if $H$ is abelian.
\end{lemma}
\begin{proof} 
By the Artin-Wedderburn theorem~\cite{dummit} and Maschke's theorem, we have a $\overline{\mathbb{Q}}$-algebra isomorphism $$\Psi: \overline{\mathbb{Q}}[H] \to \Mat_{n_1}( \overline{\mathbb{Q}})\times \cdots \times \Mat_{n_d}( \overline{\mathbb{Q}}).$$   Then under this isomorphism $\Psi$, the element $\alpha$ is sent to a $d$-tuple of matrices $(Y_1,\ldots ,Y_d)$.  Observe that $\Psi$ induces a $\overline{\mathbb{Q}}$-algebra isomorphism
between the power series rings 
$ \overline{\mathbb{Q}}[H][[t]]$ and $$\left( \Mat_{n_1}( \overline{\mathbb{Q}})\times \cdots \times \Mat_{n_d}( \overline{\mathbb{Q}})\right)[[t]]$$ and under this isomorphism $\sum \alpha^n t^{n}$ is sent to
the series $$\sum_{n\ge 0} (Y_1^n,\ldots ,Y_d^n) t^{n}.$$  This series satisfies a linear recurrence of length
$n_1+\cdots +n_d$ by the Cayley-Hamilton theorem, and thus the series is the power series expansion in $t$ of a series of the form
$P(t)/Q(t)$ with $P$ and $Q$ polynomials with coefficients in $\overline{\Q}$ and $Q(0)\neq 0$ and $\gcd(P,Q)=1$ and ${\rm deg}(Q)\le \sum n_i$ and ${\rm deg}(P) \le \sum n_i -1$.  By rescaling, we may assume that $Q(0 )=1$.  Since $F(t)$ has integer coefficients, $P/Q$ must be invariant under the action of ${\rm Gal}(\overline{\mathbb{Q}},\mathbb{Q})$ and so since $Q(0)=1$, we see that $P$ and $Q$ have rational coefficients.  
Now notice that $Q(t) F(t) = P(t)$.  We can show that the roots of $Q(t^{-1})$ must be algebraic integers and so $Q(t)$ is an integer polynomial and we can then have $P$ is an integer polynomial, since $P=FQ$ and $F$ and $Q$ have integer coefficients. 

Finally, if $\deg P+1=|H|$ or $\deg Q=|H|$, then we must have $\sum n_i=|H|$.  Since $\sum n_i^2=|H|$, this can only occur when each $n_i=1$, and so $H$ must be abelian.
\end{proof}

We can build equations that are simpler than those that arise from the
direct combinatorial interpretation using free probability. The following equation relates the inverse Cauchy Transforms of ${\alpha}$ and the ${\alpha_i}$ (see \cite[Theorem 12.7]{FPcombs}): 

\begin{equation}\label{eq:free}
\invCT_{\alpha}(t) = \left(\sum_{i=1}^{r} m_i \invCT_{\alpha_i}(t)\right) - \left(m_1+\cdots +m_r-1\right) \inv{t}.
\end{equation}

For $i=1,\ldots ,r$, we let $\Delta_i$ denote the sum of the degrees
of the irreducible representations of $G_i$. 
Let $G_1,\ldots ,G_r$ be finite groups with cogrowth series
     $F_i(t):=P_i(t)/Q_i(t)$ respectively, and let $\Delta_i = \max\{\deg (P_i)+1, \deg(Q_i)\}$.
 By definition, $\invCT_{\alpha_i}(\inv{t})$ and $tF_i(t)$ are compositional inverses, so
$$\inv{t}=\invCT_{\alpha_i}(t)\cdot P_i(\invCT_{\alpha_i}(t))/Q_i(\invCT_{\alpha_i}(t)).$$  
In particular, $\invCT_{\alpha_i}(t)$ is a root of the polynomial 
\begin{equation}\label{eq:H}
\lambda_i(t,z):=Q_i(z)  - tz P_i(z)\in \overline{\mathbb{Q}}[t,z].
\end{equation}
It follows that $\deg_z(\lambda_i)= \Delta_i$, and $\deg_t \lambda_i=1$, so $\invCT_{\alpha_i}(t)$ is algebraic.



Hence $\invCT_{\alpha_i}(t)$ lies in a field extension $E_i$ of $\overline{\mathbb{Q}}(t)$ of degree at most $\Delta_i$.
Consequently, $\invCT_{\alpha}(t)$ lies in the compositum, $E$, of $E_1,\ldots ,E_s$, which is an extension of degree at most \begin{equation}
\Delta:=\Delta_1\cdots \Delta_r.
\end{equation}

We now bound the height\footnote{If the cogrowth series, $F(t)$, has a minimal polynomial $\eta(t,z)$, We refer to $\deg_t\eta$ and $\deg_z\eta$ as the height and degree of $F(t)$ respectively.} and degree of $F(t):=F_{G;S}(t)$.
Consider the $\overline{\mathbb{Q}}(t)$-vector space $V$ with basis $e_{j_1,\ldots ,j_r}$ with $0\le j_i<\deg_z\lambda_i= \Delta_i$ for $i=1,\ldots ,r$.  Then we have a surjective $\overline{\mathbb{Q}}(t)$-linear map $\Phi$ from $V$ to $E$, given by 
$$\Phi(e_{j_1,\ldots ,j_r}) = \prod_{i=1}^r \invCT_{\alpha_i}(t)^{j_i}$$ for $1\le i\le r$, $0\le j_i< \Delta_i$. 
Notice that by Equation (\ref{eq:H}) we have
\begin{equation}
\label{eq:D'}
\sum_{\ell=0}^{\Delta_i} c_{i,\ell}(t) \invCT_{\alpha_i}(t)^{\ell} = 0,
\end{equation} where each $c_{i,\ell}(t)=[z^{\ell}]\lambda_i$ has degree at most $1$, $c_{i, \Delta_i}\neq 0$, and $c_{i,0}(t)=1$ since $Q_i(0)=1$.  
Notice that multiplication by $\invCT_{\alpha_i}(t)$ induces a $\overline{\mathbb{Q}}(t)$-linear endomorphism $L_i$ of $E$, which using Equation (\ref{eq:D'}) we can lift to an endomorphism $\bar{L}_i$ of $V$ via the rule
\begin{equation}\label{eq:endLiofV}
 \bar{L}_i(e_{j_1,\ldots ,j_s}) = \left\{ \begin{array}{ll} e_{j_1,\ldots ,j_i+1,\ldots ,j_s} & {\rm if}~ j_i < \Delta_i-1, \\
\sum_{\ell=0}^{\Delta_i-1} -c_{i,\ell}(t)c_{i,\Delta_i}(t)^{-1} e_{j_1,\ldots ,j_{i-1}, \ell, j_{i+1},\ldots ,j_s} & {\rm if}~j_i=\Delta_i-1. \end{array} \right. 
\end{equation}
Then by construction, we have $$\Phi\circ \bar{L_i}= L_i\circ \Phi$$ for $i=1,\ldots ,r$.
Now let $L$ denote the linear endomorphism of $E$ induced by multiplication by $\invCT_{\alpha}(t)$.  Then 
$L=  \sum_{i=1}^{r} m_i L_i - (m_1+\cdots +m_r-1)\inv{U}$, where $U:= (h \in E \mapsto t\cdot h\in E)$ denotes multiplication by $t$.
Notice we can lift $U$ to an endomorphism $\bar{U}$ of $V$ by the rule $\bar{U}(e_{j_1,\ldots ,j_r}) = t\cdot e_{j_1,\ldots ,j_r}$ for all $j_1,\ldots ,j_r$ and we obtain $\Phi\circ \bar{U}^n =U^n\circ \Phi$ for each $n\in \Z$.
Let
$\bar{L}:= \sum_{i=1}^{s} m_i \bar{L}_i - (m_1+\cdots +m_r-1) \inv{\bar{U}}$, then an induction argument, using linearity of the maps involved, gives $\Phi\circ C(\bar{L}) = C(L)\circ \Phi$ for every polynomial $C$.
In particular, the minimal polynomial of $L$ divides the characteristic polynomial of $\bar{L}$, since $\Phi$ is surjective.  

Let $Z$ denote the matrix of $\bar{L}$ with respect to the basis $\{e_{j_1,\ldots ,j_r}\}$.  For a fixed column indexed by $e_{j_1,\ldots ,j_r}$, let $\Gamma\subseteq \{1,\ldots ,r\}$ denote the set of $i$ for which $j_i=\Delta_i-1$. Equation~\eqref{eq:endLiofV} implies that the matrix, $Y:=Z+(m_1+\cdots +m_r-1) \inv{t} {\bf I}$ has entries of the form $a(t)/\prod_{i\in \Gamma} c_{i,\Delta_i}(t)$, where ${\rm deg}(a(t)) \le |\Gamma|$. 
Thus, the characteristic polynomial of $Y$ is a monic polynomial in $x$ of degree $\Delta:=\prod \Delta_i$ in which the coefficient of $x^k$ in the characteristic polynomial is a sum of terms that are $(\Delta-k)$-fold products of rational functions of the form $a(t)/\prod_{i\in \Gamma} c_{i,\Delta_i}(t)$, where ${\rm deg}(a(t))\le |\Gamma|$.  In particular, for a fixed $i$, since the number of standard basis elements, $e_{j_1,\ldots ,j_r}$ for which $j_i=\Delta_i-1$ is $\Delta/\Delta_i$, we see that
the coefficients of $\det(x{\bf I}-Y)$ are of the form $b(t)/\prod_{i=1}^d c_{i,\Delta_i}(t)^{\Delta/\Delta_i}$, with ${\rm deg}(b) \le  \sum_i \Delta/\Delta_i$. Considering
$$M(t):=\prod_{i=1}^d c_{i,\Delta_i}(t)^{\Delta/\Delta_i},$$ we deduce that
the characteristic polynomial of $Y$ is expressible as 
$$x^{\Delta} + \sum_{\ell=0}^{\Delta-1} {b_{\ell}(t)\over M(t)} x^{\ell},$$ where each $b_{\ell}(t)$ has degree at most $\sum_i \Delta/\Delta_i$.
Thus, the characteristic polynomial of $Z$ is of the form 
$$(x+(m_1+\cdots +m_r-1) \inv{t})^{\Delta} + \sum_{\ell=0}^{\Delta-1} {b_{\ell}(t)\over M(t)} (x+(m_1+\cdots +m_r-1)\inv{t})^{\ell}.$$
It follows that $R(t,\invCT_{\alpha}(t))=0$, where 
$$R(t,z):=M(t) (tz+(m_1+\cdots +m_r-1))^{\Delta} + \sum_{\ell=0}^{\Delta-1} t^{\Delta-\ell} b_{\ell}(t) (tz+(m_1+\cdots +m_r-1))^{\ell}.$$
Notice that $\deg_t R\le \Delta+\Delta'$, where $\Delta':=\sum_i \Delta/\Delta_i$; and $\deg_z R = \Delta$. By the change of variables $\inv{t}=uF(u)$, we deduce that
$$R(uF(u),\inv{u})=0.$$  Thus $F(t)$ is the root of the polynomial 
$$R_0(t,z):= z^{\Delta}R(tz, z^{-1}).$$ 
 It follows from standard polynomial manipulation, that $\deg_t R_0, \deg_z R_0\le \Delta+\Delta' = \Delta+ \sum_{i=1}^r \Delta/\Delta_i$.
 
 In particular, we have the following theorem.
 \begin{theorem}
   \label{thm:generalbound}
   Let $G_1,\ldots ,G_r$ be finite groups with generating set $S_1,\hdots, S_r$ respectively, and cogrowth series $F_{G_i;S_i}(t)=P_i(t)/Q_i(t)$, where $P_i,Q_i\in \Z[t]$ are polynomials with constant term $1$. For each $i=1,2,\hdots, r$, let
   $\Delta_i:=\max\{1+\deg P_i, \deg Q_i\}$.
 Then the cogrowth series $F(t):=F_{G;S}(t)$ of $G=\coprod G_i^{\star m_i}$, is algebraic and satisfies $\Lambda(t,F(t))=0$, where $\Lambda(t,z)\in \Z[t,z]$ with 
 ${\rm deg}_t(\Lambda)$ and ${\rm deg}_z(\Lambda)$ both at most
 $$\left(\prod_{i=1}^r \Delta_i\right) \left( 1 + \sum_{i=1}^r \frac{1}{\Delta_i} \right).$$
 In particular, the degrees do not depend on $m_1,\ldots ,m_r$ when we choose
 $S$ as above.
 \end{theorem}

Observe that each $\Delta_i$ is at most the degree sum of the irreducible representations of $G_i$. We therefore immediately obtain Theorem~\ref{thm:mainbound}.

%
\begin{proof}[Proof of Theorem~\ref{thm:mainbound}]
The real-valued function on $(0,\infty)^r$ given by 
$$(y_1,\hdots,y_r)\mapsto (y_1\hdots y_r) (1+\inv{y_1}+\hdots+\inv{y_r})$$ 
is increasing in each of the $y_i$. Hence, the result follows directly from Lemma~\ref{lem:bound} and Theorem~\ref{thm:generalbound}.
\end{proof}

The following remark shows how one can directly apply Theorem~\ref{thm:mainbound} to determine
explicit algebraic equations satisfied by the cogrowth series.

  
\begin{remark} If $A(t)\in \mathbb{C}[[t]]$ is a power series that is a solution to $\eta(t,A(t))=0$ where $\eta(t,z)$ is an irreducible polynomial whose degrees in $t$ and $z$ are bounded by $\Delta_t$ and $\Delta_z$ respectively
, then Bostan et al.~\cite{bostandeaf} show that $A(t)$ is $D$-finite and that it is annihilated by a differential operator $$\sum_{i=0}^k p_i(t) \partial_t^i$$ with $k\le \Delta_z$ and ${\rm deg}(p_i) \le ((2k-1)\Delta_z +2k^2 -4k+3)\Delta_t -k(k-1)/2$.
By the irreducibility of $\eta$ and Theorem~\ref{thm:mainbound}, we can consider 
$$\Delta_t=\Delta_z=\prod_{i=1}^r \Delta_i \left( 1 + \sum_{i=1}^r 1/\Delta_i \right).$$ 
It follows that the above differential operator has coefficients in $\C[t]$, each of degree at most
$$N:=((2{\Delta_z}-1){\Delta_z} + 2{\Delta_z}^2 -4{\Delta_z}+3){\Delta_z} - {\Delta_z}({\Delta_z}-1)/2 = (8{\Delta_z}^3 - 11{\Delta_z}^2+7{\Delta_z})/2.$$
Therefore, the coefficients $a_n$ of $A(t)$ satisfy a polynomial linear recurrence of the form
$$\sum_{i=0}^{N+{\Delta_z}} q_i(n) a_{n-i} = 0$$ with $q_i(x)$ of degree at most ${\Delta_z}$, and so one can theoretically ``guess-and-prove'' the recurrences with the $q_i$ and sufficient many terms in $\{a_n\}$.  
\end{remark}


\subsection{Determining the cogrowth via methods from free probability}
\label{sec:AlgSys}
In some cases it is straightforward to derive the minimal polynomial
$\Lambda(t,z)$ for the cogrowth series of a virtually free group. The proof of Theorem~\ref{thm:generalbound} gives an
outline of how to do this. In the case of cyclic factors, Liu~\cite{liu} gave a slight improvement to Theorem~\ref{thm:generalbound}. We now illustrate how to compute the
polynomial equation satisfied by the cogrowth series for the group $G=(\mathbb{Z}/d\mathbb{Z})^{\star m}$. This is a different
  approach from that used in the proof of Corollary \ref{cor:main3}, although we obtain the same conclusion. We again let $x_1,\ldots ,x_m$ denote generators for the copies of $\mathbb{Z}/d\mathbb{Z}$ and we let $S=\{x_1,\ldots ,x_m\}$.  

Using the notation in \S3.1, we see by Equation \ref{eq:free} that
$$\invCT_{\alpha}(t) = m \invCT_{\alpha_1}(t) - (m-1)t^{-1},$$ where $\alpha_1=x_1 \in \mathbb{C}[\mathbb{Z}/d\mathbb{Z}]$ and $x_1$ is a generator for $\mathbb{Z}/d\mathbb{Z}$.  
Since $\CT_{\alpha_1}(t) = \inv{t}\cdot 1/(1-t^{-d})$

By Equation (\ref{eq:inv}), we see that
$$t^{-1}  = \invCT_{\alpha_1}(t)^{-1} \cdot 1/(1-\invCT_{\alpha_1}(t)^{-d})$$
In particular, 
$z=\invCT_{\alpha_1}(t)$ is a solution to the equation
$$(z^d-1)  = z^{d-1}t,$$ and since $\invCT_{\alpha_1}=m^{-1} \invCT_{\alpha} + (m-1)/m$, we have
that 
$$(m^{-1} \invCT_{\alpha}(t)+(m-1)m^{-1}t^{-1})^d - 1 = (m^{-1} \invCT_{\alpha}(t)+(m-1)m^{-1}t^{-1})^{d-1}t.$$
We now let $t=\CT_{\alpha}(u)$ and we see
$$(m^{-1} u^{-1} + (m-1)m^{-1}\CT_{\alpha}(u)^{-1})^d -1 = (m^{-1} u^{-1} + (m-1)m^{-1}\CT_{\alpha}(u)^{-1})^{d-1} \CT_{\alpha}(u).$$
Letting $x=u^{-1}$ and using Equation (\ref{eq:11}), we see
$$(m^{-1} x + (m-1)m^{-1}x^{-1} F_{G,S}(x)^{-1})^d - 1 = (m^{-1} x + (m-1) m^{-1} x^{-1} F_{G,S}(x)^{-1})^{d-1} x F_{G,S}(x).$$
 In particular, after simplifying we see that $z=F_{G,S}(t)$ is a solution to $\Lambda(t,z)=0$, where 
 $$\Lambda(t,z) = m^d t^d z^d - (z-1)(z+m-1)^{d-1},$$
which is consistent with the result obtained in Example~\ref{exam:1} via the use of combinatorial grammars.  The language theoretic approach, however, has the added advantage of giving a mechanism for describing the language $\mathcal{L}(G,S)$.  
\section{A gap result for radii of convergence}
\label{sec:gap}
In this section we prove Theorem \ref{thm:main3}.  We first prove an elementary estimate.
\begin{lemma}
  Let $\phi: G\to H$ be a group homomorphism and let $S$ be a
  symmetric generating set for $G$.  If the restriction of $\phi$ is
  injective on $S$ then ${\rm CL}(n;G,S)\le {\rm CL}(n;H,\phi(S))$.
\end{lemma}
\begin{proof}
Observe that if $s_1,\ldots, s_n\in S$ and $s_1\cdots s_n=1$ then $\phi(s_1)\cdots \phi(s_n)=1$ and so the inequality is immediate.
\end{proof}
\begin{proof}[Proof of Theorem \ref{thm:main3}]
Suppose that $H$ is a group with symmetric generating set $S$ and let $s_1,\ldots ,s_p$ be the elements of order $2$ in $S$ and let 
$u_1^{\pm 1},\ldots ,u_q^{\pm 1}$ be the remaining elements of $S$.  Then $p+2q=|S|$.  We claim that if either $p\ge 3$, $q\ge 2$ or $p,q\ge 1$ then $\rho_{H,S}^{-1}\ge 2\sqrt{2}$.  To do this, we deal with a few cases.

\emph{Case I:} $p\ge 3$.
Let $G$ be the free product of $3$ copies of $\mathbb{Z}/2\mathbb{Z}$ with generators $x_1,x_2,x_3$
Then we have a group homomorphism $\phi:G\to H$ sending $x_i\to s_i$ for $i=1,2,3$ and this is injective on
$T:=\{x_1,\ldots ,x_3\}$.  Thus
${\rm CL}(n;G,T)\le {\rm CL}(n;H,S)$ and hence $1/\rho_{G,T} \le  1/\rho_{H,S}$.
By Theorem \ref{thm:main2} (a), with $d=2,m=3$, we have that the cogrowth generating function for $G$ with respect to $T$ is
$$4/\left(1+3\sqrt{1-8t^2}\right),$$ which has radius of convergence ${\left(2\sqrt{2}\right)}^{-1}$ and so $\rho_{H,S}^{-1}\ge 2\sqrt{2}$ in this case. 

\emph{Case II:} $q\ge 2$.
In this case, we let $G$ be the free product of two copies of $\mathbb{Z}$ with generating set $T=\{y_1,y_1^{-1}, y_2,y_2^{-1}\}$.  Then we have a homomorphism $\phi: G\to H$ sending $y_i\to u_i$ for $i=1,2$.  Then we again have
${\rm CL}(n;G,T)\le {\rm CL}(n;H,S)$ and hence $1/\rho_{G,T} \le  1/\rho_{H,S}$. Taking $m=0$ and $s=2$ in Theorem \ref{thm:main2} (b), we have that the cogrowth series of $G$ with respect to $T$ is given by the series $3/(1+2\sqrt{1-12x^2})$ using the work of Chomsky and Sch\"utzenberger~\cite{chomsky} (see also OEIS A035610).  This series has $1/\rho_{H,S}\ge 1/\rho_{G,T} =\sqrt{12}>2\sqrt{2}$ and so we get the result in this case.
\vskip 2mm
\emph{Case III}: $p,q\ge 1$.
In this case, we let $G$ be the free product of $\mathbb{Z}/2\mathbb{Z}$ (with generator $x$) with $\mathbb{Z}$ (with generating set $y,y^{-1}$.  We let $T$ be the symmetric generating set $\{x,y,y^{-1}\}$ and we have a homomorphism from $G\to H$ sending $x$ to $s_1$, $y\mapsto u_1$. Then this is injective on $T$ and sends $T$ into $S$, so
${\rm C}(n;H,S)\ge {\rm CL}(n;G,T)$, and Theorem \ref{thm:main2} (b) gives that the cogrowth generating series for $G$ has radius of convergence $1/2\sqrt{2}$, so we get the result in this case.

We see that it suffices to consider the case when $p\le 2$, $q\le 1$, and $pq=0$, and  hence $(p,q)\in \{(2,0), (1,0), (0,1)\}$.  In this case, we see that $H$ is a homomorphic image of either $D_{\infty}$ or $\mathbb{Z}$, and hence it is amenable and so by Kesten's criterion $\rho_{H,S}^{-1}=|S|\in \{1,2\}$.  The result follows.
\end{proof}
We pose the following question.
\begin{question} Does there exist $\alpha\in [2\sqrt{2},\infty)$ that cannot be realized as $1/\rho_{G,S}$ for some finitely generated group $G$ and finite symmetric generating set $S$?
\end{question}



\bibliographystyle{plain}
\bibliography{references}

\begin{thebibliography}{10}

\bibitem{alkauskas}
Giedrius Alkauskas.
\newblock The modular group and words in its two generators.
\newblock {\em Lith. Math. J.}, 57(1):1--12, 2017.

\bibitem{anisimov}
A.~V. Anisimov.
\newblock Group languages.
\newblock {\em Kibernetika}, 4:18--24, 1971.

\bibitem{BeMi20}
Jason Bell and Marni Mishna.
\newblock On the complexity of the cogrowth sequence.
\newblock {\em Journal of Combinatorial Algebra}, 4(1):73--85, 2020.

\bibitem{bostandeaf}
Alin Bostan, Fr\'{e}d\'{e}ric Chyzak, Bruno Salvy, Gr\'{e}goire Lecerf, and
  \'{E}ric Schost.
\newblock Differential equations for algebraic functions.
\newblock In {\em Proceedings of the 2007 International Symposium on Symbolic
  and Algebraic Computation}, ISSAC '07, page 25–32, New York, NY, USA, 2007.
  Association for Computing Machinery.

\bibitem{aecf}
Alin Bostan, Frédéric Chyzak, Marc Giusti, Romain Lebreton, Grégoire Lecerf,
  Bruno Salvy, and Éric Schost.
\newblock {\em Algorithmes Efficaces en Calcul Formel}.
\newblock Frédéric Chyzak (auto-édit.), Palaiseau, sep 2017.
\newblock 686 pages. Imprimé par CreateSpace. Aussi disponible en version
  électronique.

\bibitem{chomsky}
N.~Chomsky and M.~P. Sch\"{u}tzenberger.
\newblock The algebraic theory of context-free languages.
\newblock In {\em Computer programming and formal systems}, pages 118--161.
  North-Holland, Amsterdam, 1963.

\bibitem{dummit}
David~S. Dummit and Richard~M. Foote.
\newblock {\em Abstract Algebra}.
\newblock John Wiley \& Sons Inc., 2004.

\bibitem{flajolet}
Philippe Flajolet and Robert Sedgewick.
\newblock {\em Analytic Combinatorics}.
\newblock Cambridge University Press, 2009.

\bibitem{Kam}
Robert~D. Gray and Mark Kambites.
\newblock On cogrowth, amenability, and the spectral radius of a random walk on
  a semigroup.
\newblock {\em Int. Math. Res. Not. IMRN}, (12):3753--3793, 2020.

\bibitem{kuksov2}
Dmitri Kouksov.
\newblock On rationality of the cogrowth series.
\newblock {\em Proc. Amer. Math. Soc.}, 126(10):2845--2847, 1998.

\bibitem{kuich}
Werner Kuich and Arto Salomaa.
\newblock {\em Semirings, automata, languages}, volume~5 of {\em EATCS
  Monographs on Theoretical Computer Science}.
\newblock Springer-Verlag, Berlin, 1986.

\bibitem{kuksov}
Dmitri Kuksov.
\newblock Cogrowth series of free products of finite and free groups.
\newblock {\em Glasg. Math. J.}, 41(1):19--31, 1999.

\bibitem{liu}
Haggai Liu.
\newblock On the cogrowth series of free products of finite groups.
\newblock Master's thesis, Simon Fraser University, 2021.

\bibitem{muller}
David~E. Muller and Paul~E. Schupp.
\newblock Groups, the theory of ends, and context-free languages.
\newblock {\em Journal of Computer and System Sciences}, 26, 1982.

\bibitem{FPcombs}
Alexandru Nica and Roland Speicher.
\newblock {\em Lectures on the Combinatorics of Free Probability}.
\newblock Cambridge University Press, 2006.

\bibitem{novikov}
Petr~Sergeevich Novikov.
\newblock On the algorithmic unsolvability of the word problem in group theory.
\newblock {\em Trudy Mat. Inst. Steklov.}, 44:3--143, 1955.

\bibitem{gfun}
Bruno Salvy and Paul Zimmermann.
\newblock Gfun: a maple package for the manipulation of generating and
  holonomic functions in one variable.
\newblock {\em ACM Transactions on Mathematical Software}, 20:163--–177,
  1994.

\bibitem{oeis}
N.~J.~A. Sloane.
\newblock The on-line encyclopedia of integer sequences.
\newblock published electronically at \url{oeis.org.}

\end{thebibliography}
\appendix
\begin{table}
{\tiny

\begin{tabular}{llll}
$G$& $\rho_G$ & OEIS & Initial terms of $CL(n; G, S)$\\\toprule
$\mathbb{Z}_2\star\mathbb{Z}_2$&$1/2$ &A126869 & $1, 0, 2, 0, 6, 0, 20, 0, 70, 0, 252, 0, 924, 0, 3432, 0, 12870, 0, 48620, 0, 184756$\\ 
$\mathbb{Z}_3\star \mathbb{Z}_3$&$\frac{{2}^{2/3}}{3}$&A047098&$1, 0, 0, 2, 0, 0, 8, 0, 0, 38, 0, 0, 196, 0, 0, 1062, 0, 0, 5948, 0, 0, 34120$\\
$\mathbb{Z}_4\star \mathbb{Z}_4$ &$\frac{{3}^{3/4}}{4}$&A107026&$1, 0, 0, 0, 2, 0, 0, 0, 10, 0, 0, 0, 62, 0, 0, 0, 426, 0, 0, 0, 3112, 0, 0, 0, 23686$\\
$\mathbb{Z}_5\star \mathbb{Z}_5 $&$\frac{{4}^{4/5}}{5}$ &A304979& $1, 0, 0, 0, 0, 2, 0, 0, 0, 0, 12, 0, 0, 0, 0, 92, 0, 0, 0, 0, 792, 0, 0, 0, 0, 7302$\\[5mm]
$\mathbb{Z}_2\star \mathbb{Z}_3$ &.5072330945 &A265434 & $1, 0, 1, 1, 1, 5, 2, 14, 13, 31, 66, 77, 240, 286, 722, 1226, 2141, 4760, 7268, 16473$\\ 
$\mathbb{Z}_2\star \mathbb{Z}_4$ &.5171996045&NEW&$1, 0, 1, 0, 2, 0, 7, 0, 22, 0, 66, 0, 209, 0, 687, 0, 2278, 0, 7612, 0$\\
$\mathbb{Z}_2\star \mathbb{Z}_5$ &.5259851993&NEW&$1, 0, 1, 0, 1, 1, 1, 7, 1, 27, 2, 77, 19, 182, 148, 379, 793, 748, 3268, 1729$\\
$\mathbb{Z}_2\star \mathbb{Z}_6$&.5333879707&NEW&$1, 0, 1, 0, 1, 0, 2, 0, 9, 0, 36, 0, 114, 0, 316, 0, 873, 0, 2636, 0$\\
$\mathbb{Z}_2\star \mathbb{Z}_7$&.5396278153&NEW&$1, 0, 1, 0, 1, 0, 1, 1, 1, 9, 1, 44, 1, 156, 2, 450, 25, 1122, 262, 2508, 1851, 5149$\\[5mm]
$\mathbb{Z}_2\star \mathbb{Z}$&$(2\sqrt{2})^{-1}$&A089022& $1, 3, 15, 87, 543, 3543, 23823, 163719, 1143999, 8099511, 57959535, 418441191$\\[5mm]
%
%
$\mathbb{Z}_2^{\star m}$& $\frac {1}{2\sqrt [2]{m-1}}$& &$1,0,m,0,2\,{m}^{2}-m,0,5\,{m}^{3}-6\,{m}^{2}+2\,m,0,14\,{m}^{4}-28\,{m}^{3}+20\,{m}^{2}-5\,m$\\
$\mathbb{Z}_3^{\star m}$&$\frac {{2}^{2/3}}{3\sqrt [3]{m-1}}$&&$1,0,0,m,0,0,m \left( 3\,m-2 \right) ,0,0,m \left( 12\,{m}^{2}-18\,m+7 \right) ,0,0$\\
$\mathbb{Z}_4^{\star m}$&$\frac {{3}^{3/4}}{4\sqrt [4]{m-1}}$&&$1,0,0,0,m,0,0,0,m \left( 4\,m-3 \right) ,0,0,0,m \left(22\,{m}^{2}-36\,m+15 \right) ,0,0,0$\\
$\mathbb{Z}_5^{\star m}$&$\frac {{4}^{4/5}}{5\sqrt [5]{m-1}}$&&$1,0,0,0,0,m,0,0,0,0,m \left( 5\,m-4 \right) ,0,0,0,0,m \left( 35\,{m}^{2}-60\,m+26 \right) ,0,0,0,0$\\\bottomrule
\end{tabular}
\smallskip

}
\caption{\small The examples considered in Section~\ref{sec:exam}. The algebraic equations satisfied by the generating functions are found in that section. Here, we use $\{x\}$ as a generating set for $\mathbb{Z}_n=\mathbb{Z}/n\mathbb{Z}=\gen{ x~|~x^n=1}$, and we use $\{x,\inv{x}\}$ for $\Z=\gen{x}$. If $S_i\subseteq G_i$ is the generating set for $G_i$, the above cogrowth series is given with respect to $S=\cup S_i\subseteq G_1\star \cdots \star G_m$}
\label{tab:examples}
\end{table}
\end{document}